\documentclass[11pt]{amsart}

\usepackage{enumerate,url,amssymb,  mathrsfs}
\usepackage{graphicx}
\newtheorem{theorem}{Theorem}[section]

\newtheorem*{lemma*}{Lemma}
\newtheorem{proposition}[theorem]{Proposition}
\newtheorem{corollary}[theorem]{Corollary}

\theoremstyle{definition}
\newtheorem{definition}[theorem]{Definition}
\newtheorem{example}[theorem]{Example}
\newtheorem{conjecture}[theorem]{Conjecture}

\theoremstyle{remark}
\newtheorem{remark}[theorem]{Remark}

\numberwithin{equation}{section}

\newcommand{\abs}[1]{\lvert#1\rvert}
\newcommand{\norm}[1]{\lVert#1\rVert}

\newcommand{\A}{\mathbb{A}}
\newcommand{\C}{\mathbb{C}}
\newcommand{\CC}{\mathscr C}

\newcommand{\DD}{\mathbb{D}}

\newcommand{\R}{\mathbb{R}}
\newcommand{\Z}{\mathbb{Z}}
\newcommand{\T}{\mathbb{T}}

\newcommand{\const}{\mathrm{const}}
\newcommand{\dtext}{\textnormal d}
\newcommand{\ds}{\displaystyle}

\DeclareMathOperator{\re}{Re}
\DeclareMathOperator{\im}{Im}

\DeclareMathOperator{\loc}{loc}
\DeclareMathOperator{\Mod}{Mod}

\def\Xint#1{\mathchoice
{\XXint\displaystyle\textstyle{#1}}%
{\XXint\textstyle\scriptstyle{#1}}%
{\XXint\scriptstyle\scriptscriptstyle{#1}}%
{\XXint\scriptscriptstyle\scriptscriptstyle{#1}}%
\!\int}
\def\XXint#1#2#3{{\setbox0=\hbox{$#1{#2#3}{\int}$}
\vcenter{\hbox{$#2#3$}}\kern-.5\wd0}}

\def\dashint{\Xint-}
\def\le{\leqslant}
\def\ge{\geqslant}
\begin{document}

\title[Doubly connected minimal surfaces]{Doubly connected minimal surfaces and extremal harmonic mappings}

\author{{\scriptsize Tadeusz   Iwaniec}}
\address{Department of Mathematics, Syracuse University, Syracuse,
NY 13244, USA and Department of Mathematics and Statistics,
University of Helsinki, Finland}
\email{tiwaniec@syr.edu}
\thanks{Iwaniec was supported by the NSF grant DMS-0800416 and Academy of
Finland grant 1128331.}

\author{{Leonid V. Kovalev}}
\address{Department of Mathematics, Syracuse University, Syracuse,
NY 13244, USA}
\email{lvkovale@syr.edu}
\thanks{Kovalev was supported by the NSF grant DMS-0913474.}

\author{{Jani Onninen}}
\address{Department of Mathematics, Syracuse University, Syracuse,
NY 13244, USA}
\email{jkonnine@syr.edu}
\thanks{Onninen was supported by the NSF grant  DMS-0701059.}


\date{February 14, 2010}

\maketitle
\begin{abstract}
The concept of a conformal deformation has two natural extensions: quasiconformal and harmonic mappings. Both classes do not preserve the conformal
type of the domain, however they cannot change it in an arbitrary way. Doubly connected domains are where one first observes nontrivial conformal
invariants.  Herbert Gr\"{o}tzsch and  Johannes C. C. Nitsche addressed this issue for quasiconformal and harmonic mappings, respectively.  Combining
these concepts we obtain sharp estimates for quasiconformal harmonic mappings between doubly connected domains. We then apply our results to the
Cauchy problem for minimal  surfaces, also known as the Bj\"orling problem.
Specifically, we obtain  a sharp estimate of the modulus of a doubly connected minimal surface that evolves
from its inner boundary with a given initial slope.
\end{abstract}

 \tableofcontents

\section{Introduction: if Gr\"otzsch had met Nitsche}

The classical Bj\"orling problem~\cite{Bj,DHKW} is to find a minimal surface that contains a given curve and
has prescribed normal vector along the curve. For real-analytic data a local solution exists and admits an integral representation due to
H.~A.~Schwarz. However, the Schwarz representation tells us little about global geometric properties of the solution, as was recently
emphasized by the authors of~\cite[p.~787]{MWe}, see also~\cite{GM} and~\cite{Mi}. In this paper we study the Bj\"orling problem for closed curves.
Such curves give rise to doubly connected parametric minimal surfaces, not necessarily embedded in $\R^3$.
A natural global invariant associated with such a surface is the conformal modulus of its domain of isothermal parametrization.
We give a sharp estimate for conformal modulus of solution of a Bj\"orling problem. Our method is based on analysis of
quasiconformal harmonic mappings in doubly connected domains. Both concepts, quasiconformality and harmonicity, are generalizations
of conformal mappings. The extremal problem that we solve in this paper has roots in the works of Herbert  Gr\"otzsch (for quasiconformal mappings)
and Johannes Nitsche (for harmonic mappings). The main results of the paper are Theorems~\ref{main1}, ~\ref{th34}, and~\ref{nongraph}.

\subsection{Expanding the Notion of Conformality}
We study mappings $h = u + i v  \colon \Omega \to \C $ defined in a domain $\Omega$  of the complex plane $\,\mathbb C = \{ z= x+iy\,,\;\; x , y \,\in \mathbb R\,\}$. The partial differentiation in $\,\Omega\,$ will be expressed by the Wirtinger operators $$\frac{\partial}{\partial z} = \frac{1}{2} \left(\frac{\partial}{\partial x} - i \frac{\partial}{\partial y}\right )\; \;\;\;\textrm {and}\;\;\;\;  \frac{\partial}{\partial \bar z} = \frac{1}{2} \left(\frac{\partial}{\partial x} + i \frac{\partial}{\partial y}\right) \;, \;  $$
Accordingly, we shall abbreviate the complex derivatives of $h$ to
\[
h_z= \frac{\partial h}{\partial z} = \frac{1}{2} \left(\frac{\partial h}{\partial x} - i \frac{\partial h}{\partial y}\right )\; \;\;\;\textrm {and}\;\;\;\;  h_{\bar z} = \frac{\partial h}{\partial \bar z} = \frac{1}{2} \left(\frac{\partial h}{\partial x} + i \frac{\partial h}{\partial y}\right)
\]
In these terms the complex differential form $\,\textrm{d}h(z) = h_z(z)\, \textrm{d}z\,+\,h_{\bar z}(z)\, \textrm{d}\bar z \,$ represents the Jacobian matrix
\[ Dh(z) = \left[\begin{array}{cc} u_x&v_x\\u_y&v_y\end{array}\right]\,,\]
Its operator norm and the Hilbert-Schmidt norm are given by:
\[
 \norm{Dh}=  \abs{h_z}+ \abs{h_{\bar z}}\,,\qquad \abs{Dh}^2\;=\;  2\,(\,|h_z|^2\,+\, |h_{\bar z}|^2\,)\; = \; u_x^2 + v_x^2 + u_y^2 + v_y^2
\]
and the Jacobian determinant
\[
J_h(z)=J(z,h) =\det Dh(z)= u_x v_y \,-\,u_y v_x \; = \abs{h_z}^2-\abs{h_{\bar z}}^2\,.
\]

\subsubsection{Quasiconformal mappings}

The concept of planar quasiconformal mappings has originated around 1928 from the paper by H. Gr\"{o}tzsch~\cite{GrU}, though the term  ``quasiconformal" was coined by Ahlfors only in 1935~\cite{Ah}.  Among many equivalent definitions used nowadays, we conveniently adopt the following analytical one:
\begin{definition}
An orientation preserving homeomorphism $\,h \colon \Omega  \rightarrow \C $ of Sobolev class $\mathscr W^{1,1}_{\loc}(\Omega , \C)$ is said to be {\it $K$-quasiconformal} , $1 \le K < \infty$ , if
\begin{equation}\label{distine}
\norm{Dh(z)}^2 \le K J(z,h)\;, \qquad \mbox{ almost everywhere }\;,
\end{equation}
or equivalently
\[
  \abs{Dh(z)}^2 \;\le \left( K + \frac{1}{K}\right) J(z,h)
\]
\end{definition}
It is worth noting that $K$-quasiconformal mappings are invariant under the conformal change of the variables in both the domain and the target space. They have positive Jacobian determinant almost everywhere, thus one can speak of the {\it distortion function}
\begin{equation}
K_h(z):=\frac{\norm{Dh(z)}^2}{\det Dh(z)} = \frac{(|h_z|  + |h_{\bar z}|)^2}{|h_z|^2 -|h_{\bar z}|^2} = \frac{|h_z|  + |h_{\bar z}|}{|h_z| -|h_{\bar z}|}\leqslant  K
\end{equation}
Note the Dirichlet energy formula for $\,h\,$ over a measurable subset $\,U\subset \Omega\,$
\begin{equation}\nonumber
   \begin{split}
   \mathscr E_{_U} [h] := &\; \iint_U \|Dh\|^2 \textrm{d}x\,\textrm{d}y \;\le\,\\&\left( K + \frac{1}{K}\right) \,\iint_U J(z,h)\, \textrm{d}x \,\textrm{d}y  = \left( K + \frac{1}{K}\right)\, |h(U)|
\end{split}\end{equation}
which is finite whenever the area $\,|h(U)|\,$ of the image of $\,U\,$ is finite. In particular, $  \,h \in \mathscr W^{1,2}_{\loc}(\Omega , \C)\,.$
Another way to express the distortion inequality ~\eqref{distine} is:
\begin{equation}\label{eq14}
\abs{h_{\bar z}(z)} \le k \,\abs{h_z(z)}\;, \qquad k=\frac{K-1}{K+1}\;< 1\,.
\end{equation}
There is a useful interplay between quasiconformal mappings and the first order elliptic systems of PDEs in the complex plane. The most general linear (over $\mathbb R$) elliptic operator for orientation preserving mappings (homotopic to the Cauchy-Riemann operator) takes the form:
$$
 \mathcal B = \frac{\partial}{\partial \bar {z}} \;-\;\mu(z)\;\frac{\partial}{\partial z} \;-\;\nu(z)\;\overline{\frac{\partial}{\partial \bar {z}}} \;:\; \mathscr W^{1,2}_{\loc}(\Omega) \rightarrow \mathscr L^2_{\loc}(\Omega)
$$
where $\,\mu\,$ and $\,\nu\,$ are complex valued measurable functions such that
$$
  |\mu(z)|\;+\;|\nu(z)|  \le \;k  =  \frac{K-1}{K+1}  \;<1\,, \;\;\;a.e. \;\Omega
$$
Two special cases, referred to as the first and the second Beltrami operators, are worth noting. It is customary to investigate geometric and analytic features of a quasiconformal mapping via the celebrated first Beltrami equation, because it is linear over the complex numbers
\begin{equation}\label{beltrami}
h_{\bar z} = \mu (z) h_z, \qquad \abs{\mu (z)} \le k <1
\end{equation}
However, we shall take advantage of the {\it second Beltrami  equation}
\begin{equation}\label{nudist}
{h_{\bar z}} = \nu (z) \overline{h_z}, \qquad \abs{\nu(z)} \le k <1\;.
\end{equation}
The \textit{$\nu$-Beltrami coefficient}
\begin{equation}\label{SBC}\nu (z)= {h_{\bar z}}(z)/{\overline{h_z (z)}}
\end{equation}
 tells us not only about quasiconformal features  but also how far is  $h$ from harmonic functions; it is harmonic exactly when  $\,\nu\,$ is antianalytic.
  One major advantage of the equations of type (\ref{nudist})  over (\ref{beltrami}) is that they are preserved upon conformal change of the $z$-variable in $\,\Omega\,$, while equations in (\ref{beltrami}) are not. The $\nu$-Beltrami coefficient will be used to describe geometric entities of minimal surfaces, such as the Gauss map, etc.

\subsubsection{Harmonic mappings}
The theory of minimal surfaces provides us with another classical example of useful generalization of conformal mappings.
These are {\it complex harmonic  functions} whose real and imaginary parts need not be coupled in the Cauchy-Riemann systems. Of special interest
to us will be orientation preserving harmonic homeomorphisms. Such mappings are $\mathscr C^\infty$-diffeomorphisms due to Lewy's
Theorem \cite[p. 20]{Dub}. For other (noninjective) harmonic functions the second Beltrami coefficient will still be defined as an antimeromorphic
function.

\subsection{The Gr\"{o}tzsch Distortion Problem (1928)}
We recall the Conformal Mapping Theorem: Every doubly connected domain $\,\Omega \subset \mathbb C\,$ can be mapped conformally onto a circular region $\,A(r,R) = \{ z\,;\; r <|z|<R\,\}$, where $\,0\leqslant r < R\leqslant \infty\,$. It will simplify the arguments, and cause insignificant loss of generality, if we restrict ourselves to doubly connected domains of finite conformal type; that is, when
 $\,0< r < R <\infty\,$. We call such $\,\Omega\,$ a \textit{ring domain}. The ring domains fall into conformal equivalence classes, according to their \textit{modulus}.
\begin{equation}
\textrm{Mod} \,\Omega  \;=\; \log \frac{R}{r}
\end{equation}
The famous Schottky theorem (1877) \cite{Sc}   asserts that an annulus
\[\A = A(r,R)= \{z \in \C \colon r < \abs{z} <R\}\]
can be mapped conformally onto the annulus
\[\A^\ast = A(r_\ast,R_\ast)= \{w \in \C \colon r_\ast < \abs{w} <R_\ast\}\]
if and only if $$\Mod \A := \log \frac{R}{r} = \log \frac{R_\ast}{r_\ast} =: \Mod \A^\ast\;;\qquad\;\textrm{that is,}\;\;
\frac{R}{r}= \frac{R_\ast}{r_\ast}$$
Moreover, modulo rotation, every conformal mapping $h \colon \A \overset{\textnormal{\tiny{onto}}}{\longrightarrow} \mathbb A^\ast$ takes the form
\[h(z) = \frac{r_\ast}{r} z \qquad \mbox{ or } \qquad h(z) = \frac{R\, r_\ast}{z}\]
Note that the latter mapping, though orientation preserving, reverses the order
of the boundary circles.
The mapping problem for doubly connected domains becomes more flexible if we admit quasiconformal deformations. However, there are still constraints on the domains.
\begin{theorem}[Gr\"otzsch (1928)] Let $h \colon \Omega \overset{\textnormal{\tiny{onto}}}{\longrightarrow} \Omega^\ast$ be a $K$-quasiconformal mapping between doubly connected domains. Then
\[\frac{1}{K} \Mod \Omega \le \Mod \Omega^\ast \le K \Mod \Omega\]
If $\Omega$ and $\Omega^\ast$ are circular annuli, say  $A(r,R)$ and $A(r_\ast,R_\ast)$,  then this inequality translates into
\begin{equation}\label{Grbd}
 \left(\frac{R}{r}\right)^{1/K}\le   \frac{R_\ast}{r_\ast} \le \left(\frac{R}{r}\right)^K
\end{equation}
Equalities are attained, uniquely modulo rotation,  for the multiples of the mappings $h(z)= \abs{z}^{\frac{1}{K}-1}z$ and $h(z)= \abs{z}^{{K}-1}z$, respectively.
\end{theorem}
The reader may wish to notice that these mappings fail to be harmonic, except for $K= 1$.
\begin{remark}
In general a homeomorphism $h \colon \Omega  \overset{\textnormal{\tiny{onto}}}{\longrightarrow} \Omega^\ast$ between doubly connected domains does not extend continuously to the closure of $\Omega$. Nevertheless it gives a one-to-one correspondence between boundary components in the sense of cluster sets.  The closed set $\partial \Omega \subset \C$ consists of two components called the inner boundary $\partial_I \Omega$ and outer boundary  $\partial_O \Omega$.
\end{remark}
\begin{definition}
We denote $\mathcal H (\Omega, \Omega^\ast )$ the class of orientation preserving homeomorphisms $h \colon \Omega  \overset{\textnormal{\tiny{onto}}}{\longrightarrow} \Omega^\ast$ which send $\partial_I\Omega$ onto $\partial_I \Omega^\ast$.
\end{definition}
It should be noted that every homeomorphism $h \colon \Omega  \overset{\textnormal{\tiny{onto}}}{\longrightarrow} \Omega^\ast$, upon suitable conformal reparametrization of $\Omega$, becomes a member of $\mathcal H (\Omega, \Omega^\ast)$. For this reason the restriction to mappings in $\mathcal H (\Omega, \Omega^\ast)$ will involve no loss of generality in our subsequent statements. One type of domains will be needed.
\begin{definition}
A {\it half circular annulus} is a doubly connected domain, denoted by $\mathcal A = \mathcal A(r, \cdot)$, whose inner boundary is the circle $\T_r =\{z \colon \abs{z}=r\}$.
\end{definition}

\subsection{The Nitsche Conjecture (1962)}
The study of doubly connected minimal surface (or minimal annuli) has been an active area for several decades, see~\cite{CM, CM2, Fa, MW} for recent
developments. The roots of this theory were founded, among others, by J.~C.~C. Nitsche who made several pivotal contributions in the 1960s. Considering
the existence of  doubly connected minimal graphs over a given annulus $\mathbb A^*$, he raised a question~\cite{Ni} of existence of a harmonic
homeomorphism between circular annuli $h \colon \A=A(r,R) \overset{\textnormal{\tiny{onto}}}{\longrightarrow} \A^\ast=A(r_\ast, R_\ast)$. He observed
that  the modulus of  $\A^\ast$ can be arbitrarily large but not arbitrarily small. Then he conjectured that harmonic homeomorphisms $h \colon \A \overset{\textnormal{\tiny{onto}}}{\longrightarrow} \A^\ast$ exist if an only if
\begin{equation}\label{nibd}
\frac{R_\ast}{r_\ast} \ge \frac{1}{2} \left(\frac{R}{r}+ \frac{r}{R}\right)
\end{equation}
The {\it Nitsche bound}~\eqref{nibd}, which appeared in~\cite{AIMb, Dub, Ka,  Lyz0, Ly, Nib, Sch, We},  was recently proved by the authors in~\cite{IKO2}. It turned out that~\eqref{nibd} also holds for harmonic homeomorphisms in the class $\mathcal H (\A, \mathcal A^\ast)$ where $\mathcal A^\ast=\mathcal A(r_\ast, \cdot) $ is a half circular annulus contained in $\A^\ast$. Moreover, the equality in~\eqref{nibd}   takes  place if and only if $\mathcal A^\ast = \A^\ast$ and (modulo rotation) $h(z)= \frac{1}{2} \left(\frac{z}{r} + \frac{r}{\bar z}\right)$. This so-called \textit{critical Nitsche mapping}  fails to be quasiconformal at the inner boundary. At this point it is worthwhile to mention that the same Nitsche bound is necessary and sufficient for the existence of a minimizer of the Dirichlet energy
\[
   \mathscr E[h] \;=\;\iint_\A  \abs{Dh(z)}^2\;\textrm{d}x\, \textrm{d}y\;=\;2\,\iint_\A \left( |h_z|^2 + |h_{\bar z}|^2\right)
\]
subject to all homeomorphisms $h \colon \A \overset{\textnormal{\tiny{onto}}}{\longrightarrow} \A^\ast$, see~\cite{AIM, IO}.

\subsection{Gr\"{o}tzsch meets Nitsche}
Our main result strengthens the estimates of Gr\"{o}tzsch and Nitsche for mappings which are both quasiconformal and harmonic. Here is the simplified version of it.

\begin{theorem}\label{main1}
Let $\,h \,: \mathbb A \overset{\textnormal{\tiny{onto}}}{\longrightarrow} \mathbb A^\ast\,$ be a $\,K$-quasiconformal harmonic homeomorphism. Then we have
\begin{equation}\label{Main1}
 \frac{R_\ast}{r_\ast} \ge  \frac{K+1}{2K}\;\frac{R}{r}\;+\; \frac{K-1}{2K}\;\frac{ r}{R}
 \end{equation}
Equality is attained if and only if, modulo rotation, $h$ takes the form
 \begin{equation}\label{Main2}
 h(z) = \frac{K+1}{2K}\frac{z}{r} + \frac{K-1}{2K} \frac{r}{\bar z}.
 \end{equation}
\end{theorem}
\begin{remark}
The estimate (\ref{Main1}) readily implies both the Gr\"{o}tzsch estimate
\[\frac{R_\ast}{r_\ast} \ge \frac{K+1}{2K}\;\frac{R}{r}\;+ \;\frac{K-1}{2K}\;\frac{ r}{R}\; \ge \;\Big(\frac{R}{r}\Big)^\frac{K+1}{2K} \Big(\frac{r}{R}\Big)^\frac{K-1}{2K}= \Big(\frac{R}{r}\Big)^{1/K},\]
by Young's inequality, and the Nitsche bound
$$\frac{R_\ast}{r_\ast} \ge \frac{K+1}{2K}\;\frac{R}{r}\;+ \;\frac{K-1}{2K}\;\frac{ r}{R}\; \ge \frac{1}{2} \left(\frac{R}{r}+ \frac{r}{R}\right)$$
\end{remark}
We conjecture the following analogue of the upper Gr\"otzsch bound for $K$-quasiconformal harmonic mappings.

\begin{conjecture}
Let $\,h : \mathbb A \overset{\textnormal{\tiny{onto}}}{\longrightarrow} \mathbb A^\ast\,$ be a $\,K$-quasiconformal harmonic homeomorphism. Then we have
\begin{equation}
 \frac{R_\ast}{r_\ast}\; \le \; \frac{K+1}{2}\;\frac{R}{r}\;- \;\frac{K-1}{2}\;\frac{ r}{R}\;\;\;\;\;\;\;\;\;\;\qquad\Big\{ \leqslant \Big(\frac{R}{r}\Big)^{\!K}\,\Big\}
\end{equation}
Equality is attained uniquely, modulo conformal automorphisms of $\A$, for  $h(z) = \frac{K+1}{2}\frac{z}{r} - \frac{K-1}{2} \frac{r}{\bar z}$.
\end{conjecture}

Our applications to minimal surfaces require a more general version of Theorem~\ref{main1} in which
quasiconformality and injectivity are imposed only on the inner boundary of the domain.
We write $\dashint$ for the integral average and
\[A[1,R]=\{z\in\C\colon 1\le \abs{z}\le R\},\quad \T_r=\{z\in\C\colon \abs{z}=r\},\quad \T=\T_1.\]
The following theorem is the principal result of our paper.

\begin{theorem}\label{th34}
Let a $\,\mathscr C^1$-mapping $\,h \,: A[1,R] \rightarrow \mathbb C\,$ be harmonic in $\,A(1,R)\,$ and satisfy the following  conditions at the inner boundary $\,\mathbb T$,
\begin{itemize}
\item $\,h :\,\mathbb T\leftrightarrows \mathbb T\,$ is an orientation preserving homeomorphism
\item  There is $\,1\leqslant K <\infty\,$ \,(the distortion of $\,h\,$  at $\,\mathbb T\,$) such that
\begin{equation}\label{311}
\norm{Dh(z)}^2 \;\leqslant\; K \,J(z,h) \;, \qquad \; \textrm{for all}\;\;\; z\in\mathbb T
\end{equation}
\end{itemize}
Then
\begin{equation}\label{312}
\sup_{|z|= R} |h(z)| \;\ge\;\;\left[ \dashint_{\T_{\!R}} \abs{h}^2  \right]^\frac{1}{2} \ge  \frac{K+1}{2K}\;R+ \frac{K-1}{2K}\;\frac{ 1}{R}
\end{equation}
Equality is attained, uniquely up to rotation, for
\[ h(z) = \frac{K+1}{2K}\;z+ \frac{K-1}{2K}\;\frac{ 1}{\bar z}\]
\end{theorem}

The smoothness assumptions on $h$ and $\Omega$ can be easily removed with an approximation argument~\cite[Lemma 2.2]{IKO2}.
Upon a conformal change of the independent variable $z$ we immediately obtain the following generalization of
Theorem~\ref{th34}.

\begin{corollary}\label{coromega}
Let $\,h \,: \Omega \rightarrow \mathbb C\,$ be a harmonic function in a doubly connected domain $\Omega\subset \mathbb C$ that is $K$-quasiconformal near the inner boundary $\partial_I \Omega$. Suppose also that   $\abs{h(z)} \searrow 1$ as $z \to \partial_I \Omega$.
Then
\[
\sup_{\Omega} \abs{h} \;\ge\;\;  \frac{1}{2K}\left[(K+1)\;e^{\,\Mod\,\Omega}+ (K-1)\;e^{-\Mod\,\Omega}\right]
\]
\end{corollary}

Corollary~\ref{coromega} can be viewed as a \emph{reverse Harnack inequality}, as it gives a sharp estimate of the ratio
\[\frac{\sup_{\Omega} \abs{h}}{\inf_{\Omega} \abs{h}}\]
from below.
It is an interesting question whether one can dispose of the assumption
that one of the boundary components is mapped homeomorphically onto $\T$. We formulate this as a conjecture.

\begin{conjecture}\label{greatconj}
Let $\Omega\subset\C$ be a doubly connected domain. Suppose
$h\colon\Omega\to \C_\circ=\C\setminus\{0\}$ is a harmonic mapping that is not homotopic to a constant
within the class of continuous mappings from $\Omega$ to $\C_\circ$. Then
\begin{equation}\label{great1}
\frac{\sup_{\Omega} \abs{h}}{\inf_{\Omega} \abs{h}}\ge \cosh \left(\frac{1}{2}\Mod\Omega\right)
\end{equation}
If $h$ is in addition injective, then the factor $1/2$ can be omitted.
\end{conjecture}

The mapping $h(z)=z+\bar z^{-1}$ attains equality in~\eqref{great1} with $\Omega=A(1/R,R)$. It also provides an
injective example  when restricted to the annulus $A(1,R)$.

\section{Minimal Surfaces}

For the most part of this article the terminology is standard, but it seems worthwhile to recall and modify the required notions.

We are dealing with  $\,\mathscr C^1$-mappings $\,F = (u,v,w) :\Omega \overset{\textnormal{\tiny{into}}}{\longrightarrow}\mathbb R^3\,$  defined in  a domain $\,\Omega\,$   of the complex plane
$\,\mathbb C = \{ z= x+iy\,,\;\; x , y \,\in \mathbb R\,\}$ and valued in the 3-space $\,\mathbb R^3\,$. Let us factor the target space  into the complex plane and the real line $\, \mathbb R^3 \simeq \mathbb C\times \mathbb R = \{(\xi, w)\;;\; \xi \in \mathbb C\,,\; w\in\mathbb R\}$. Thus $F = (h, w) \colon \Omega \rightarrow  \mathbb C\times \mathbb R $,  where $ h =  u+ iv  $ will be referred to as  \textit{complex coordinate} and $\,w\,$  as \textit{real coordinate} of $\,F\,$.
The Riemann sphere $\, \widehat{\mathbb C} = \,\mathbb C \cup \{\infty\}$ is the inverse  image of  $\mathbb S^2 = \{(\xi, w)\;;\; |\xi|^2 + w^2 = 1\} \subset \mathbb R^3 $ under the stereographic map

\[
\mathcal S \colon \widehat{\mathbb C} \overset{\textnormal{\tiny{onto}}}{\longrightarrow} \mathbb S^2\;,\;\;\;\;\; \mathcal S(z)  = \left( \frac{2z}{1+|z|^2}\;,\frac{|z|^2 -1}{1+|z|^2}\right) \;\in \mathbb C\times \mathbb R
\]

\subsection{Parametric Surfaces}

We define an oriented parametric surface $\, \Sigma \,$ in $\,\mathbb R^3\,$ to be an equivalence class of mappings  $\,F = (u, v, w) : \Omega \rightarrow \mathbb R^3\,$ of some domain $\,\Omega \subset \mathbb R^2\,$ into $\,\mathbb R^3\,$, where the coordinate functions $\, u =u(x,y), \,v = v(x,y) \; \textrm{and} \; w = w(x,y)\,$ are of class at least $\,\mathscr C^1(\Omega)\,$. Two such mappings $\,F = (u, v, w) : \Omega \rightarrow \mathbb R^3\,$ and $\,\tilde{F} = (\tilde{u}, \tilde{v}, \tilde{w}) : \tilde{\Omega} \rightarrow \mathbb R^3\,$, referred to as  parametrizations of the surface, are said to be equivalent if there is a $\,\mathscr C^1$-diffeomorphism $\,\phi :\tilde{\Omega}\overset{\textnormal{\tiny{onto}}}{\longrightarrow} \Omega \,$ of positive Jacobian determinant  such that $\; \tilde{F} = F\circ \phi\,$.  Let us call such $\,\phi\,$ a \textit{change of variables, or reparametrization} of the surface. Furthermore, we assume that the critical points of  $\,\Sigma\,$  are isolated. These are the points  $\,(x,y) \in \Omega\,$  at which  the tangent vectors $\,F_x = \frac{\partial F}{\partial x},\; F_y = \frac{\partial F }{\partial y}\,$ are linearly dependent. Equivalently, at the critical points the \textit{Jacobian matrix}

\begin{displaymath}
\mathfrak DF(x,y)  = \left[\begin{array}{ccc} u_x&v_x&w_x \\u_y&v_y&w_y\end{array}\right]
\end{displaymath}
has rank at most 1. It has full rank  2 at the \textit{regular points}.
   Various geometric entities associated with a given surface will be introduced with the aid of a parametrization, but in fact they will be invariant under reparametrization. By Implicit Function Theorem, all parametrizations  of the surface are locally injective near regular points, though self intersections in $\,F(\Omega)\,$ may occur.
A surface with no critical points is called an \textit{immersion}.
The cross product of $\,F_x\,$ and $\,F_y\,$ at a regular point represents nonzero normal vector to the surface:
  \begin{displaymath}
   F_x \times F_y  =  \left(\left|\begin{array}{cc} v_x & w_x \\ v_y & w_y \end{array}\right|\,,\; - \left|\begin{array}{cc} u_x & w_x \\ u_y & w_y \end{array}\right|\,,\;\left|\begin{array}{cc} u_x & v_x \\ u_y & v_y \end{array}\right|\right)\;\neq \;0
  \end{displaymath}
 The area of the surface equals
\[
 \left| F(\Omega)\,\right|  \, =\, \iint_\Omega \left| F_x\times F_y\right|\; \textrm{d}x\, \textrm{d}y
\]
The central geometric entity is the \textit{normal vector field}, also known as the \textit{Gauss map}:
\begin{equation}
 N :\Omega \rightarrow \mathbb S^2 \subset \mathbb R^3\,, \;\;\;\;\; N(x,y) = \frac{F_x\times F_y}{|F_x\times F_y| }
\end{equation}
This map defines \textit{spherical image} $\,N(\Omega)\subset \mathbb S^2\,$ of the surface, a subset of the unit sphere.

\subsection{Isothermal parameters}

 In what follows, we will concern ourselves mostly with conformal parametrizations  $\,F = (u, v, w) : \Omega \rightarrow \mathbb R^3\,$. This simply means that the coordinate functions,  called \textit{isothermal parameters}, will satisfy the conformality relations:

 \begin{equation}\nonumber
 \;\begin{split}
  \;\begin{cases} u_xu_y  +v_xv_y + w_x w_y  =  0  \,, \quad\quad\quad\quad\; \;\;\;\,(\,F_x \;\textrm{and} \; F_y \; \textrm{are orthogonal in} \;\;\mathbb R^3\,)\\
  u_x^2  + v_x ^2 + w_x^2  =  u_y^2 + v_y^2 +w_y^2 \;\,   \,\;\quad\quad\; \; (\,F_x \;\textrm {and} \;F_y\;\;\textrm{have egual length\;)}
 \end{cases}\end{split}
 \end{equation}
 Equivalently, it means that:
\begin{equation}
 |F_x\times F_y | \;= \; |F_x|\cdot|F_y|\; = \;|F_x|^2\;  = |F_y|^2 \;\;
 \end{equation}
Thus $\,F\,$ is an  immersion if $\,\mathfrak DF \,\neq 0$ at every point. We refer to~\cite{DHKW} for an excellent historical account of existence of isothermal coordinates.
When dealing with conformal mappings we should take advantage of the complex variables. The conformality relations reduce to one complex equation
\[
 u_z^2  + v_z ^2 + w_z^2  = 0\,,\;\;\; \textrm{where}\;\;  (u_z, v_z, w_z) = \frac{\partial}{\partial z} F \,=  F_z  \in \mathbb C^3
\]
 Recall the notation $\, \mathbb R^3 \simeq \mathbb C\times \mathbb R$ and  $F = (h, w) \colon \Omega \rightarrow  \mathbb C\times \mathbb R $,  where $ h =  u+ iv  $ is a complex coordinate of $\,F\,$.
 A simple direct computation shows that $u_z^2  +  v_z ^2  \, = (u_z + i v_z ) ( u_z - i v_z) = \, h_z \,\overline{h_{\bar z}}$. Hence the conformality relations simplify even further to $\, h_z \,\overline{h_{\bar z}}\;+\,w_z^2 \;= 0$. We shall try to express surfaces in terms of their complex isothermal coordinate $\,h :\Omega \rightarrow  \mathbb C\,$ without appealing to its real coordinate $\,w\,$. This is possible in view of the following proposition.
\begin{proposition}
Let $\,h : \Omega \rightarrow \mathbb C\,$ be the complex coordinate of the isothermal representation $F = (h, w) \colon \Omega \rightarrow  \mathbb C\times \mathbb R $  of a surface. Then
\begin{itemize}
\item the function $ \, h_z \overline{h_{\bar z}}\;$ admits a continuous branch of square root in $\,\Omega\,$.
\item  for each smooth closed curve $\,\Gamma\subset \Omega\,$ we have
\begin{equation}\label{Zero}     \im \int_\Gamma \sqrt{h_z \overline{h_{\bar z}}} \;\textrm d z \; =\; 0
\end{equation}
\item the real isothermal coordinate is given by
\begin{equation} \label{for w}
     w = 2\, \im \int_{z_\circ}^ z \sqrt{h_z \overline{h_{\bar z}}} \;\textrm d z
\end{equation}
where the line integral runs along any smooth curve $\,\gamma\subset \Omega\,$ beginning  at a given point $\,z_{\circ} \in \Omega\,$ and terminating at $\,z\,$.\\
\item The normal vector field at a regular point of a  surface is
\begin{equation}
  N(z) = \frac{F_x\times F_y}{|F_x\times F_y| }\;=\; \left(\frac{2i\, \sqrt{h_z h_{\bar z}} }{|h_z| +|h_{\bar z}|}\;,\; \frac{|h_z|  - |h_{\bar z}|}{|h_z| +|h_{\bar z}|} \right) \;\in \mathbb S^2 \subset \mathbb C\times \mathbb R
 \end{equation}
\end{itemize}
\end{proposition}
\begin{proof}
The latter formula is a matter of a simple direct computation. We reserve the following notation for  the normal vector $\,N(z) =  ( \xi(z),\,\tau(z))\,$, where the complex component $\,\xi :\Omega \rightarrow \mathbb C\,$ and the real component $\,\tau :\Omega \rightarrow \mathbb R\,$ satisfy $\,|\xi|^2 + \tau^2 = 1\,$.
The only not obvious fact is that for each smooth closed curve $\,\Gamma\subset \Omega\,$
\begin{equation}\nonumber
   \begin{split}  - 2\,\im \int_\Gamma \sqrt{h_z \overline{h_{\bar z}}} \;\textrm d z \; = &\;2\,\re \int_\Gamma  w_z \,\textrm{d}z =
\int_\Gamma  \,( w_z \,\textrm{d}z \, + \,\overline{w_z \,\textrm{d}z}\;) = \\ &\int_\Gamma  \,( w_z \,\textrm{d}z \, + \,w_{\bar{z}}\, \textrm{d}\bar{z}) = \int_\Gamma \textrm{d} w \;= \;0\,.
 \end{split}
   \end{equation}
 \end{proof}

 \begin{remark}

If $\,h \,$ is harmonic in a doubly connected domain $\,\Omega\,$  then the integrand $\,\sqrt{h_z \overline{h_{\bar z}}}\,$ is holomorphic. In this case one needs to verify (\ref{Zero})  only for one closed curve homologous to the boundary components.
\end{remark}

\subsection{Doubly Connected Surfaces}
We adapt to our use the following definition:

\begin{definition}
 An open doubly-connected surface  $\,\Sigma\,$ in $\,\mathbb R^3\,$ is a conformal immersion $\,F = (u, v, w) : \Omega \rightarrow \mathbb R^3\,$
 in which $\,\Omega\,$ is either:
\begin{itemize}

\item{punctured complex plane  $\mathbb C_\circ  =  \{ z \in \mathbb C ;\; z \neq 0\}$}
\item {punctured disk $\mathbb D_\circ = \{ z \in\mathbb C ;\; 0< |z| < 1\}$}
\item {or an annulus  $\mathbb A = A(r,R) =   \{z;\; r<|z| < R \}\;$ ,  $\;\;\;0<r <R< \infty$ }

\end{itemize}
\end{definition}

In this latter case, referred to as of finite conformal type, we define the \textit{conformal modulus} of  $\,\Sigma\,$ (or briefly, \textit{modulus}) by seting
\begin{equation}
    \textrm {Mod}  \mathbb \;\Sigma  =  \log \frac{R}{r}   > 0
\end{equation}
The classical theorem of  Schottky~\cite{Sc} tells us that the ratio $ \frac{R}{r} $ is independent of the conformal parametrization.  Even more, the images of the radial segments of the annulus  and the images of the concentric circles  are independent of the isothermal parametrization.

 \subsection{Graphs} The term {\it graph} over a domain $\,\Omega^\ast \subset \mathbb C\,$ refers to the parametric surface of the form: $\,F =(h, w) \,:\;\Omega \rightarrow \mathbb C\times \mathbb R\,$ in which $\,h \colon \Omega  \overset{\textnormal{\tiny{onto}}}{\longrightarrow} \Omega^\ast $ is a $\,\mathscr C^1$-diffeomorphism and  $\,w(z)  =  \Phi (h(z))\,$,  where $\,\Phi = \Phi(\xi)\,$ is a real-valued function in $\,\xi\in \Omega^\ast\,$. Graphs are always regular surfaces; the normal vector field has nonvanishing real coordinate.

\subsection{Minimal Surfaces}
The study of multiply connected minimal surfaces has a long history~\cite{Ba,DHKW,Kau,Ni1,Ni3,OsSc}.
A parametric surface is {\it minimal} if and only if the isothermal parameters are harmonic or, equivalently, the complex vector field $ F_z \colon \Omega \rightarrow  \mathbb C^3 $ is holomorphic (Enneper-Weierstrass representation). Thus $F = (h, w) \colon \Omega \rightarrow  \mathbb C\times \mathbb R $  is represented by a complex harmonic map  $ h =  u+ iv \colon \Omega \rightarrow  \mathbb C $ and a real harmonic function $\, w \colon \Omega \rightarrow  \mathbb R $. In addition these functions are coupled by the conformality relation:

 \begin{equation}
  h_z\, \overline{h_{\bar z}} \;+\; w_z^2\;\equiv 0\;\qquad \; h_{z\bar z} = w_{z\bar z} \;\equiv 0
 \end{equation}

Note that any conformal change of the $\,z$-variable ( analytic bijective map) leads to equivalent minimal surface  in different isothermal parameters.
All zeros of the holomorphic function $ \, h_z \overline{h_{\bar z}}\;$ have even order. The real isothermal parameter is determined, at least locally, in terms of $\,h$ as
\begin{equation}\label{w}
   w  =  2 \,\im \int \sqrt{\, h_z \,\overline{h_{\bar z}}\;} \;\textrm{d} z
\end{equation}

\begin{remark}\label{remark1}
It is evident that every complex harmonic homeomorphism $\,h :\Omega \rightarrow \mathbb C\,$  can be lifted,  locally near every zero of even order of $ \, h_z \overline{h_{\bar z}}\;$, to isothermal parameters of a minimal surface. The surface has  $ w \equiv \const $  if and only if $h$ is holomorphic or antiholomorphic. The global lifting exists provided the imaginary part of the integral in (\ref{w}) is single valued. This is a question which we often encounter when dealing with doubly connected minimal surfaces. Perhaps the best examples of this are catenoid and helicoid over an annulus, see Section~\ref{sec28}.
\end{remark}
\subsection{Area}
The area formula for a parametric surface in isothermal coordinates reduces to the Dirichlet energy of $\,F\,$,
\[
 \left| F(\Omega)\,\right|  \, =\, \frac{1}{2}\,\iint_\Omega \abs{D F}^2\;  = \, \frac{1}{2}\,\iint_\Omega \left( |F_x|^2  + |F_y|^2 \right)\; \textrm{d}x\, \textrm{d}y
\]
If the surface is minimal, we find that
\begin{equation}\nonumber
\begin{split}
&|F_x|^2  = |F_y|^2 = \frac{1}{2} \left( |F_x|^2 + |F_y|^2 \right) = \frac{1}{2}  ( u_x^2 +v_x^2 + w_x^2 + u_y^2 +v_y^2 +w_y^2 ) \\
& =\;|h_z|^2 + |h_{\bar z}|^2   +  2 |w_z|^2   =  |h_z|^2 + |h_{\bar z}|^2  +  2 |h_z|\cdot|h_{\bar z} |  =  \left( |h_z|  + |h_{\bar z}| \right )^2
\end{split}
\end{equation}
Hence
\begin{equation}
|F_x|  = |F_y| =  |h_z|  + |h_{\bar z}|  \,=\,  \norm{Dh(z)}  \;\;-\textrm{the operator norm}
\end{equation}
Now the area formula simplifies further in terms of $\,h\,$:
\[
 \left| F(\Omega)\,\right|  \, =\, \iint_\Omega (|h_z|  + |h_{\bar z}|)^2\; \textrm{d}x\, \textrm{d}y
\]

A word of caution, the variational equation for this latter integral, when considered for all homeomorphisms $\,h \colon \Omega  \overset{\textnormal{\tiny{onto}}}{\longrightarrow} \Omega^\ast $ , is not the Laplace equation.

\subsection{The Second Beltrami Equation}

 Suppose the minimal surface $ F = (h,w):\Omega \rightarrow  \mathbb C \times \mathbb R\,$,  in isothermal parameters, has $w\not\equiv \const$. Consequently, the zeros of the holomorphic function $\,h_z\,$ are isolated. This yields what we call the \textit{second Beltrami equation} for $h :$
\begin{equation}
{h_{\bar z}}  =  \nu(z)\, \overline{h_z} \;,\;\;\; \;\;\nu = \bar{\lambda}^2(z)\;,\;\;\quad\;\quad\;\textrm{and} \;\;\;w_z  =  \pm\;i\,\lambda \; h_z\,
\end{equation}
where $\,\lambda :\Omega \rightarrow \mathbb C\,$ is a meromorphic function. The $\nu$-coefficient  $\,\nu(z) = \overline{\lambda}\,^2(z)\,$, will be viewed as a known quantity, whereas $\,h\,$ as one of many possible solutions to this equation. Nonetheless, the Gauss normal field turns out to be independent of the solution, namely
\begin{equation}\label{211}
  N(z) = \left(\,\xi(z),\;\tau(z)\right) = \; \left(\frac{2i\, \overline{\lambda } }{1 +|\lambda|^2}\;,\; \frac{1  - |\lambda|^2}{1 +|\lambda|^2} \right) \;\in \mathbb S^2 \subset \mathbb C\times \mathbb R
 \end{equation}

Observe that for any minimal graph (with $\,h\,$ orientation preserving) the normal vectors $\,N(z)\,$ belong to   the northern hemisphere; their vertical components  $\,\tau = \tau(z)\,$ are positive. By virtue of~\eqref{211},  the vertical components  measure quasiconformality of $\,h \colon \Omega  \rightarrow \mathbb C $ at the given point $\,z\in \Omega\,$.
For example, if a minimal graph  is obtained by lifting a $\,K$-quasiconformal harmonic map $\,h : \Omega \rightarrow \mathbb C\,$ then its spherical image lies in the cap
\[\overset{\frown }{\mathbb S}_{\!\!_K} = \{(\xi,\tau)\, ; \;|\xi|^2 + \tau^2 = 1\,,\;\; \tau \geqslant \frac{1}{K}\}\,\subset\,\mathbb S^2\]

\subsection{Catenoid and Helicoid}\label{sec28}
The best known minimal surfaces are the catenoid and helicoid.  The catenoid $\,F =(h,w) :
\mathbb C_\circ \rightarrow  \mathbb C \times \mathbb R\,$  is furnished by
the parameters:
\begin{equation}\label{catenoid}
  h(z) =  \frac{1}{2} \left( z + \frac{1}{\bar z}\right ) ,\;\;\;\;\;  w(z)  = \log |z| \,,\;\;
\end{equation}
The complex coordinate map $\,h :\mathbb C_\circ\,\overset{\textnormal{\tiny{onto}}}{\longrightarrow}\,\;\mathbb C\setminus \mathbb D $ folds along the unit circle, where the second Beltrami equation changes its ellipticity status
\begin{equation}\label{Beltr}
h_{\bar z} = -\bar z^{-2} \overline{h_z}
\end{equation}
\[
|\nu(z)| =\frac {1}{|z|^2} \;
\begin{cases}
&<1 \;,\; \textrm{if}\;\; |z|>1\;,\;\;\textrm{orientation preserving}\\
&>1 \;, \; \textrm{if}\;\; |z| <1\;,\;\;\textrm{orientation reversing}
 \end{cases}
\]
 The Gauss map is precisely the stereographic projection $\,N :\mathbb C_\circ \overset{\textnormal{\tiny{into}}}{\longrightarrow}\, \mathbb S^2\,$.
\begin{equation}
 N(z)  = \left( \frac{2z}{1+|z|^2}\;,\frac{|z|^2 -1}{1+|z|^2}     \right)\;,\;\;\;\;\;\;\;\; z \in \mathbb C_\circ
\end{equation}
Thus $\,N :\mathbb C_\circ \overset{\textnormal{\tiny{into}}}{\longrightarrow}\, \mathbb S^2\,$ is a one-to-one conformal map which omits the north and the south poles of the sphere.  By way of illustration, here is another solution to the same equation (\ref{Beltr}), which gives rise to the isothermal coordinates for  Enneper's surface
\[
 h(z) =  \bar{z} \;-\frac{1}{3} \,z^3 \;,\;\;\;\;\;\ w(z)  =  \re \,z^2\, , \;\;\;\;\;\;\; z\in \mathbb C
\]
This time the Gauss map omits only the north pole of the Riemann sphere. By way of digression, Catenoid and Enneper's surfaces are the only complete regular minimal surfaces whose normal map is one-to-one~\cite[p. 87]{Os1}. The reader may wish to verify that $\,h(z) =  \bar{z} \;-\frac{1}{3} \,z^3\,$ is injective in the unit disk.

Let us now return to the catenoid. This time  $\,h= h(z)\,$ is restricted to an annulus $\,\A \;= A(1,R)\;= \left\{ z\in \C \colon \;1\;< \abs{z}\;<R \;\, \right\}\,$
\begin{equation}
   h : \mathbb A  \overset{\textnormal{\tiny{onto}}}{\longrightarrow} \mathbb A^\ast \,,\;\;\;\;\;\; \mathbb A^\ast= A(1,R^\ast),\;\;\;\; R^\ast = \frac{1}{2}\left( R +\frac{1}{R}\right )
\end{equation}
This is a doubly connected minimal graph over $\mathbb A^\ast\,$. Noteworthy is that a different type minimal surface will emerge if we
change sign of the $\nu$-coefficient at (\ref{Beltr}). This produces  a helicoid  expressed by locally defined conformal parameters as follows:
\begin{equation}\label{helicoid}
  h(z) =  \frac{1}{2} \left( z - \frac{1}{\bar z}\right ) ,\;\;  w(z)  = \pm\,\arg \,z\, \,,\;\;\;\textrm{thus}\;\; \overline{\nu(z)} = z^{-2}\;, \;\;z\in\mathbb C_\circ
\end{equation}
To obtain single valued global parametrization one needs only replace $\,z \in\mathbb C_\circ\,$ with   $\, e^z\,, \,z \in\mathbb C \,$. Such uniformization of the parameter yields the familiar global isothermal representation of a helicoid
\begin{equation}
  h(z) =  \frac{1}{2} \left( e^z - e^{-\bar z} \right ) ,\;\;\;\;\;  w(z)  = \,\im \,z\, \,,\;\; \quad\; z\in\mathbb C
\end{equation}

\subsection{Principal Harmonics} The two complex harmonic functions that we conferred about in (\ref{catenoid}) and (\ref{helicoid}) will be useful. Let us reserve for them special notation,
 $$ \,h^\sharp(z) =  \frac{1}{2} \left( z + \frac{1}{\bar z}\right )\;\;\;\;\;\;\;\;\;\;\;\;\; \,h^\flat(z) =  \frac{1}{2} \left( z - \frac{1}{\bar z}\right) \;,\;\;\;\; z = \rho e^{i\theta} $$
 We have the following Dirichlet and Neumann boundary conditions on the unit circle $\,|z| = 1\,$:
 \begin{equation}\nonumber
 \begin{split}
  \begin{cases} \;\;|h^\sharp(z)| \equiv 1 \;\;,\;\;\;\;\;\;\;\;\; \left|\frac{\partial}{\partial \rho} \,h^\sharp (z)\right | \equiv 0\\\\
  \;\; |h^\flat(z)| \equiv 0\;\;,\;\;\;\;\;\;\;\;\;\left|\frac{\partial}{\partial \rho}\, h^\flat(z) \right| \equiv 1
\end{cases}\end{split}
 \end{equation}
More generally, to a given nonnegative real number  $\,\upsilon\,$ there corresponds an orientation preserving harmonic mapping
\begin{equation}
\begin{split}
   &h^\upsilon (z) =  h^\sharp(z) \;+ \;\upsilon\, h^\flat(z) = \frac{1}{2} \left( z + \frac{1}{\bar z}\right ) \;+\;\frac{\upsilon}{2} \left( z - \frac{1}{\bar z}\right )\\
   &=\;\frac{1+\upsilon}{2} z + \frac{1-\upsilon}{2}\,\frac{1}{\bar z} = \;\left(\frac{1+\upsilon}{2} \rho + \frac{1-\upsilon}{2}\,\frac{1}{\rho}\right) e^{i\theta} \;,\;\;\;\;\;\;\;\; z = \rho\, e^{i\theta}
 \end{split}
 \end{equation}
Thus $\,h^1\,$ is the identity mapping. For all $\upsilon \geqslant 0\,$, direct computation shows that at the unit circle, $\,|z| = 1\,$,  we have;
\begin{equation}\label{Initial}
  \;h^\upsilon(z)  = z \;,\quad
  \;\;\frac{\partial}{\partial \rho} h^\upsilon (z)   =  \upsilon \,z \;, \quad
  \frac{\partial}{\partial \rho} |h^\upsilon (z)|   =  \upsilon\;,\;\;\;
 \end{equation}
The $\nu$-coefficient of $\,h^\upsilon\,$ at  $\,z \in\mathbb C_\circ\,$  equals
$$\;\nu (z) \;= \;\nu(z,\upsilon)\; =\;\frac{\upsilon -1}{\upsilon +1}\, \frac{ 1}{\bar{z}^2}\, $$
When $\,\upsilon\,$ is strictly positive, we have uniform ellipticity in the second Beltrami equation outside the unit disk,
\[
 |\nu(\rho\,e^{i\theta})| \leqslant |\nu(e^{i\theta})| \;= \frac{|1-\upsilon |}{1+\upsilon}\, < 1\;,\;\;\;\;\;\textrm{for} \;\;\;\rho\, > 1
\]
Thus $\,h^\upsilon\,$ is $\,K$-quasiconformal in $\,\mathbb C\setminus \overline{\mathbb D}\,$, where  $\, K  = \max\{\upsilon, \upsilon^{-1}\}\,$. In the critical case of $\,\upsilon = 0\,$ the ellipticity of the Beltrami equation is lost at the unit circle where the Jacobian of $\,h^\sharp\,$ vanishes. The harmonic mappings $\,h^\upsilon \,: \mathbb C_\circ \rightarrow \mathbb C \,$ give rise to a catenoid if $\,0\leqslant \upsilon \,< 1\,$, the punctured plane if $\,\upsilon = 1\,$,  and helicoid if $\,\upsilon > 1\,$. Outside the unit disk they are self-homeomorphisms  $\,h^\upsilon : \mathbb C \setminus \mathbb D \rightleftarrows\, \mathbb C\setminus \mathbb D $.  When restricted to an annulus $\,A(1,R) = \{z ; 1<|z|<R\}\,$, these mappings turn out to be extremal for numerous questions concerning minimal surfaces. From now on, to capture only doubly connected surfaces, we shall confine ourselves to the parameters:
\begin{equation}
                   0 \leqslant \;\upsilon\;\leqslant 1
\end{equation}
One may view $\,h^\upsilon\,$ as function of circles defined by the rule $\, h^\upsilon(\mathbb T_\rho) = \mathbb T_{\rho(\upsilon)}\,$, where $\,\rho(\upsilon) = \frac{1+\upsilon}{2}\,\rho +\frac{1-\upsilon}{2} \,\frac{1}{\rho} \;. $ Let us paraphrase this view by calling this function \textit{harmonic evolution} of the inner boundary $\,\mathbb T =\mathbb T_1\,$.
 By virtue of equations (\ref{Initial}) parameter $\,\upsilon\,$ represents the initial rate of change of the radii of the circles. Therefore, the term \textit{initial speed of the evolution} should be attached to $\,\upsilon$.

\section{Bj\"orling Problem for Minimal Surfaces}
There are several natural geometric problems that lead to minimal surfaces. The classical Plateau problem, for example, is the question of finding
a surface framed by one or several Jordan curves with minimum area. Thus the boundary components of the surface are given and fixed.
This problem is a geometric counterpart of the Dirichlet problem for elliptic equations.
Another classical approach to creation of minimal surfaces goes back to Bj\"orling~\cite{Bj}. In it a minimal surface emanates from a given Jordan
curve under prescribed slope at every point of the curve, and continues to grow (on one or both sides of the curve) until unacceptable singularities occur.
This is very much reminiscent of the classical initial value problem for curves of a second order ODE.

Note that in this concept we are looking for an isothermal parametrization of a minimal surface having prescribed Dirichlet and Neumann boundary values. This leads us to the familiar illposed Cauchy problem for the Laplace equation.  Precisely, we asks for the harmonic extension of a given self-homeomorphism $\,h :\mathbb T\leftrightarrows \mathbb T\,$ of the unit circle $\,\mathbb T\,$ which takes outward  the concentric circles  $\,\{\mathbb T_\rho\}_{1\leqslant \rho < 1+\epsilon}\,$  into Jordan curves with prescribed initial speed, i.e. with given normal derivative at $\,\mathbb T\,$. In this course of
\textit{harmonic evolution of circles} the images of $\,\mathbb T_\rho\,$ when lifted to the minimal surface become isothermal latitude curves.

The illposed Cauchy Problem for elliptic equations (thus non-characteristic setting) has a long and distinguished history. On the one hand the celebrated Cauchy-Kovalevskaya Theorem asserts that Cauchy problems on non-characteristic analytic varieties have unique solutions if the Cauchy data and the coefficients of the partial differential equation are real analytic functions. On the other hand the instability (the lack of continuous dependence of the solutions on the data, even in the real-analytic case), first shown in the famous example by Hadamard \cite{Had2}, demonstrates illposedness of the problem. In fact Hadamard's early work \cite{Had1, Had2} stimulated the theory, see \cite{Lav1, Lav2, Lav3, LRS, Pli}. Analogous questions for minimal surfaces will certainly
gain in interest if we enhance them with geometric interpretations. While analyzing a given Cauchy data, one should bear in mind all its geometric features for the well-posedness. For example the subsequent successful global solutions and their stability seem to depend on
the total energy allowed for the evolution.


Given a Jordan curve  $\,\Gamma\subset \mathbb R^3\,$, it is natural to ask whether there is a regular minimal surface passing through this curve whose normal vectors (Gauss map) are prescribed on $\,\Gamma\,$. In other words the slope of the surface is given at every point of $\,\Gamma\,$.  The two principal conditions must be imposed. The first of these, a geometric one, is that the given normal vectors must be orthogonal to $\,\Gamma\,$. The second, regularity, is that both $\,\Gamma\,$ and the normal vector field must admit a real-analytic parametrization.
In all that follows, we assume  that $F_\circ : \T \rightarrow \mathbb R^3\,$ is a given real analytic one-to-one map, called \textit{parametric Jordan curve},
\[
  F_\circ =F_\circ (e^{i\theta})\;,\;\;\;\;\;\;\frac{\textrm d}{\textrm d \theta } F_\circ( e^{i\theta})  \neq 0 \,,\;\;\;\textrm{for}\;\;\;\;\;\; 0\leqslant \theta < 2\pi
\]
In addition, we shall choose and fix a real analytic vector field  $N_\circ \colon\T \rightarrow \mathbb S^2 \subset \mathbb R^3 \simeq \mathbb C\times \mathbb R \,$ that is orthogonal to $F_\circ$, meaning that
\[
 \left\langle N_\circ(e^{i\theta})\;,\; \frac{\textrm d}{\textrm d \theta } F_\circ( e^{i\theta})  \right\rangle \;\;=  \;0\,,\;\;\;\textrm{for all}\;\;\;\;\;\; 0\leqslant \theta < 2\pi
\]
We shall call such a pair $(F_\circ, N_\circ)$ the real-analytic \textit{Bj\"orling data}.
\begin{definition}(Real-analytic setting) Given the real-analytic Bj\"orling data $(F_\circ, N_\circ)$, extend $F_\circ$ to a minimal surface $F \colon A(r,R) \to \R^3$, for some  $r<1<R$,  whose Gauss map $N(z)= N_\circ (e^{i\theta})$ at  $z=e^{i \theta}$.
\end{definition}
A grasp of our goals is obtained when one has in mind a doubly connected minimal surface $\,\Sigma\,$ that is regular near the unit circle but not far from it.

\subsection{An Example}

To expect here that a regular surface $\,\Sigma\,$, which emanates from a real analytic Jordan curve, will remain regular for the whole process of evolution is entirely unrealistic. Such a situation is illustrated by the following example
\begin{example} Consider the second Beltrami equation in the annulus $A(\frac{3}{4},R)$, $R>1$
\[
   {h_{\bar z}} \;=\; \nu(z) \, \overline{h_z} \,, \;\;\;\;\;\;\; \textrm {where} \;\; \nu(z) = \frac{1}{4\,\bar{z}^4}\,, \;\;\;\;|\nu(z)| < 1
\]
The following solution, together with the associated third isothermal parameter $\,w = w(z)\,$,  represent a minimal surface
 $$h(z) = \frac{1}{15} \left( 16\, z - \frac{1}{\bar z} \right) \;+\; \frac {4}{45} \left( z^3 \,-\frac{1}{\bar z^3}\right)\,,\;\;\;\;\;\;\;\; w(z) = \frac{4}{15}  \im \left( z - \frac{4}{z} \right)$$
 Thus $\,h(e^{i\theta}) = e^{i\theta}\,$ and  we have
\[
  h_z(z) = \frac{4( 4 + z^2)}{15}  \,,\;\;\;\;h_{\bar z}(z) = \frac{ ( 4 + \bar{z}^2)}{15 \;\bar{z} ^4}\,,\;\;\;\;w_z = \frac{-2i( 4 + z^2)}{15\; z^2}
\]
 Hence the conformality relation  $ h_z \,\overline{h_{\bar z}} \, +\, w_z^2  =  0 $ is readily verified. The Jacobian determinant of $\,h\,$ is positive

\[
 |h_z|^2  \,-\,|h_{\bar{z}} |^2  \; \;\geqslant  \;\left|\frac{ 4 + z^2}{15} \right|^2  \;> 0\,,
\]
except for two critical points  of the  parametrization at  $\,z = \pm \,2i\,$, where  we have $h_z = h_{\bar z} = w_z = 0$. Let us examine $h$ in the half closed annulus $A[1,R)$. The evolution of the inner circle begins in a regular fashion with a positive speed
\[
  \frac{\partial }{\partial \rho } \left|h(\rho e^{i\theta})\right | _{\;\textrm{at}\;\rho = 1}\;\; =  \;\frac{17}{15} \;+ \frac{8}{15}
  \,\cos 2\theta \;\geqslant\; \frac{3}{5}
\]
We obtain a minimal graph over an annulus $\, A(1, \sigma)\,$, with $\sigma >1$ sufficiently close to $1$.  The initial average speed equals
\[
  \upsilon :=  \dashint_{\T}|h|_\rho  = \;\frac{17}{15}
\]
Nevertheless, far from the unit circle the injectivity of $\,h\,$ is lost. Even more, the map $\,h\,$ returns to its initial values on the inner circle, it  even vanishes at $\, z = \pm i \,\lambda\,$,
\[h(\pm i \,\lambda) = 0, \;\;\;\;\textrm{where} \;\;\;   \lambda ^6 - 48 \lambda ^4 + 3 \lambda ^2  - 4  = 0\,,\;\;\;\lambda \approx 6.92\]
\end{example}

\subsection{Existence and uniqueness for the Bj\"orling problem}

\begin{proposition} \label{cauchyprop}
To each real analytic Bj\"orling data $(F_\circ, N_\circ)$, with $N_\circ \colon \T \overset{\textnormal{\tiny{into}}}{\longrightarrow} \overset{\frown }{\mathbb S}_{\!\!_K}$, there corresponds unique minimal surface $F \colon A(r,R) \to \mathbb C \times \mathbb R$, defined for some $r<1<R$, such that $\, F(z) =  F_\circ (e^{i\theta})\,$ and its Gauss map $\,N(z)=N_\circ (e^{i\theta})$ for $\,z = e^{i\theta}\,$.
\end{proposition}
\begin{proof}
The proof is immediate from Cauchy-Kovalevskaya's theorem. We shall, nevertheless, give some details for better insight.

We look for $F : A(r,R) \rightarrow \mathbb C\times \mathbb R\,$ in the form $\,F(z)=\big(h(z),w(z)\big)\,$,
where $h \colon A(r,R ) \to \C$ and  $w \colon A(r,R ) \to \R$ are harmonic functions. They are know at the unit circle. Namely,  $F_\circ (e^{i \theta})= \big(h_\circ (e^{i \theta}), w_\circ (e^{i \theta})  \big)$, so we have the initial real analytic values.
\begin{equation}\label{Fzero}
\begin{split}
h(e^{i \theta}) &= h_\circ (e^{i \theta})\\
w(e^{i \theta}) &= w_\circ (e^{i \theta})
\end{split}
\end{equation}
The complex harmonic function $h$ must satisfy a second Beltrami equation
\begin{equation} \label{Be} h_{\bar z} = \nu (z) \overline{h_z}\end{equation}
where $\nu$ is antianalytic in $A(r,R)$. The values $\nu (e^{i\theta})$ are uniquely determined by knowing $N_\circ (e^{i \theta})$,
\begin{equation}\label{aa}N_\circ (e^{i \theta})  = \left(\frac{2 i \sqrt{\nu}}{1+ \abs{\nu}} \; , \; \frac{1-\abs{\nu}}{1+\abs{\nu}}  \right) \;= (\xi , \tau) \in \mathbb C\times \mathbb R
\end{equation}
 In particular, $\nu$ is real-analytic on $\T$, it has single valued square root, and $\abs{\nu(e^{i \theta})} \le k <1$. The latter condition  allows us to solve (\ref{Be}) for $\,h_\rho\,$ in terms of $\,h_\theta = \frac{\textrm d\,h_\circ}{\textrm {d}\theta}\,$ on the unit circle.

  Therefore, there exists unique harmonic function $\,h = h(z)\,$ in some annulus $A(r,R)$, $r<1<R$, such that
\begin{equation}\label{number1}
h(z)= \sum_{n \in \Z} \left(a_n z + \frac{b_n}{\bar z^n}\right) + c \log \abs{z},
\end{equation}
where the coefficients $a_n$ and $b_n$ are determined  from the Fourier series
\[
\begin{split}
h_\theta (e^{i \theta}) &= i \sum_{n \in \Z} n(a_n+b_n) e^{i n \theta} \\
h_\rho (e^{i \theta}) &= \sum_{n \in \Z} n(a_n-b_n) e^{i n \theta} + c
\end{split}
\]
Concerning the real isothermal coordinate $\,w\,$, we observe that its derivatives must satisfy

\begin{equation}\label{ss}
\begin{split}
&\re (\overline{\xi}\, h_\rho ) \;+\; \tau w_\rho = 0 \;,\;\;\;\;\;\;(\xi, \tau)  = N \in \widehat{\mathbb S}_K\,,\;\;\;\tau >\frac{1}{K} \\
& \re (\overline{\xi}\, h_\theta ) \;+\; \tau w_\theta = 0 \;,\;\;\;\;\;\;(\xi, \tau)  = N \in \widehat{\mathbb S}_K\,,\;\;\;\tau >\frac{1}{K}
\end{split}
\end{equation}
These relations simply express the fact that $\,N\,$ is orthogonal to the tangent vectors $\,F_\rho = (h_\rho , w_\rho )\,$ and $\,F_\theta = (h_\theta , w_\theta)\,$. As before, knowing $\,w\,$ and $\,w_\rho\,$ on $\,\mathbb T\,$ determines uniquely the harmonic function $\,w = w(z)\,$ in an annulus $\,A(r,R)\,$. It remains to verify the conformality relation $\,h_z\, \overline{h_{\bar z}} \;+\; w_z^2\;\equiv 0\;$ in the annulus.
We see from (\ref{ss}) that on the unit circle  $ -2 \tau \,w_z \,=\, \overline{\xi}\,h_z \;+\xi\, \overline{ h_{\bar z}}  $. Hence
\[
 4\,\tau^2\,w_z^2  \;=\; \overline{\xi}\,^2 h_z ^2 \;+\; \xi\,^2 \overline {h_{\bar z}}\,^2 \;+\; 2\, |\xi|\,^2\, h_z \overline {h_{\bar z}}
\]
where we recall from (\ref{aa}) that $\,\xi^2\, = \frac{- 4\,\nu \,}{(1+|\nu|)^{2}} \;$ and $\; \tau ^2 =  \frac{(1-|\nu|)^2}{ (1+|\nu|)^{2}}\,$, and that $\; \nu = h_{\bar z}/\overline {h_z}\,$. This yields the conformality relation $\,h_z\, \overline{h_{\bar z}} \;+\; w_z^2\;\equiv 0\;$ on the unit circle. Since the left hand side of this equation represents a complex analytic function in an annulus $\,A(r,R)\,$, by unique continuation property,  the conformality relation remains valid in this annulus.
\end{proof}

\subsection{Alternative formulation}
The Bj\"orling Problem can be formulated many different ways by using various geometric terms as the Cauchy data.  By way of illustration, suppose we are given  a real analytic one-to-one parametric Jordan curve    $F_\circ (e^{i \theta})= \big(h_\circ (e^{i \theta}), w_\circ (e^{i \theta})  \big)$  together with the slope of the surface at every point of the curve. We take the vertical coordinate $\,\tau\,$ of the normal vector field $\,N_\circ( e^{i\theta}) \, = (\xi, \tau) \in \mathbb C \times \mathbb R \,$ as representative of the slope of the surface, and assume that $\,0<\tau \leqslant 1 \,$.  This
 determines the modulus of the $\nu$-coefficient of the minimal surface at the unit circle, see~\eqref{aa}
\[
    k = \,k(z)\, =\,|\nu(z)| =  \frac{1-\tau(z)}{1+\tau(z)} < 1
\]
The minimal surface  in an annulus $\,A(r,R)\,,  \,r<1<R\,$,  is obtained by solving the following system of equations for complex function $\,h\,$ and real function $\,w\,$.

\begin{equation}\label{2e}
\;\begin{split}\;
\begin{cases}
h_z \,\overline{h_{\bar z}} \;+\; w_z^2  = 0 \;,    \qquad h_{z \bar z} \equiv 0,  \quad w_{z \bar z}   \equiv 0 \\
|h_{\bar z} |  =  k  |h_z| \,;\qquad 0\leqslant k = k(z)  < 1
\end{cases}
\end{split}
\end{equation}
Here we regard as known quantities the tangential derivatives $\,h_\theta\,$, $\,w_\theta\,$ and $\, k(z)\,$ at  $\,z= e^{i\theta}\,$. Let us assume that $\,h\,$ takes the unit circle $\,\mathbb T\,$ diffeomorphically onto a Jordan curve, so $\,h_\theta \neq 0\,$. The following compatibility inequality for the tangential derivatives must be imposed
\begin{equation}\label{cmpatibility}
\frac{|w_\theta|}{|h_\theta|}  \leqslant   \sqrt{ K^2 - 1} \;,\quad K = K(z) = \frac{1+k(z)}{1-k(z)}\ge 1
\end{equation}
This is immediate from the second equation in (\ref{ss});
\[
  \left|\frac {w_\theta}{h_\theta }\right| \; =\; \left|\frac{\re(\bar{\xi} h_\theta)}{ \tau\, h_\theta}\right| \leqslant \frac{|\xi|}{\tau}  = \sqrt{ \frac{1}{\tau^2}  - 1 } \leqslant \sqrt{ K^2 -1}
\]
In polar coordinates the system~\eqref{2e} takes the form

\begin{equation}\label{2epolar}
\;\begin{split}\;
\begin{cases}
(h_\rho - i h_\theta) \,( \overline{h_\rho}  - i \overline{h_\theta})  \;+\; (w_\rho - i w_\theta )^2  = 0 \;,     \\
|h_\rho + i h_\theta |  =  k  |h_\rho - i h_\theta|
\end{cases}
\end{split}
\end{equation}
Straightforward computation reveals that

\begin{equation}
w_\rho  \;=\; \frac{1}{K}  \sqrt{(K^2 -1) |h_\theta|^2  \;-\; w_\theta ^2 }\,,\qquad\,\;\;\;\; h_\rho = (a + i b )\, h_\theta
\end{equation}
where
\[
a\,= \frac{-\,w_\theta}{K\,|h_\theta|^2}  \sqrt{(K^2 -1)|h_\theta|^2  - w_\theta^2}\;\;\;\;\;\textrm{and}\;\;\;
 b = - \frac{1}{K}\left( 1 + \frac{w_\theta^2}{|h_\theta|^2}\right)\; \leqslant 0
\]
The latter inequality,  $ \, \im\,\frac{h_\rho}{h_\theta} \;= b \leqslant 0\;$,  has a geometric significance. It tells us that at the unit circle the vectors $\,h_\rho\,$ are directed outward the Jordan curve $ h(\mathbb T)\,$
if $h(\T)$ is traversed counterclockwise. In particular, if $\,h(\mathbb T) = \mathbb T\,$ then
\[
\abs{h}_\rho  = -\abs{h_\theta}\, \im\,\frac{h_\rho}{h_\theta}    \ge \frac{\abs{h_\theta}}{K} \ge 0
\]
 Let us illustrate this computation in two examples.

First consider the ellipse $\,F_\circ( e^{i\theta}) = ( e^{i\theta},    \lambda \cos \theta ) \,$ circumscribed on the cylinder $\,\mathbb T\times \mathbb R\,$, where we chose $ \lambda \ne 0 $ to be constant on $\,\mathbb T\,$. The minimal surface is flat if its initial slope coincides with that of the plane of the ellipse; that is, when  $\,  \lambda  = \sqrt{\tau^{-2} -1}\,$.

\begin{center}
\begin{figure}[h!]
\includegraphics[height=70mm]{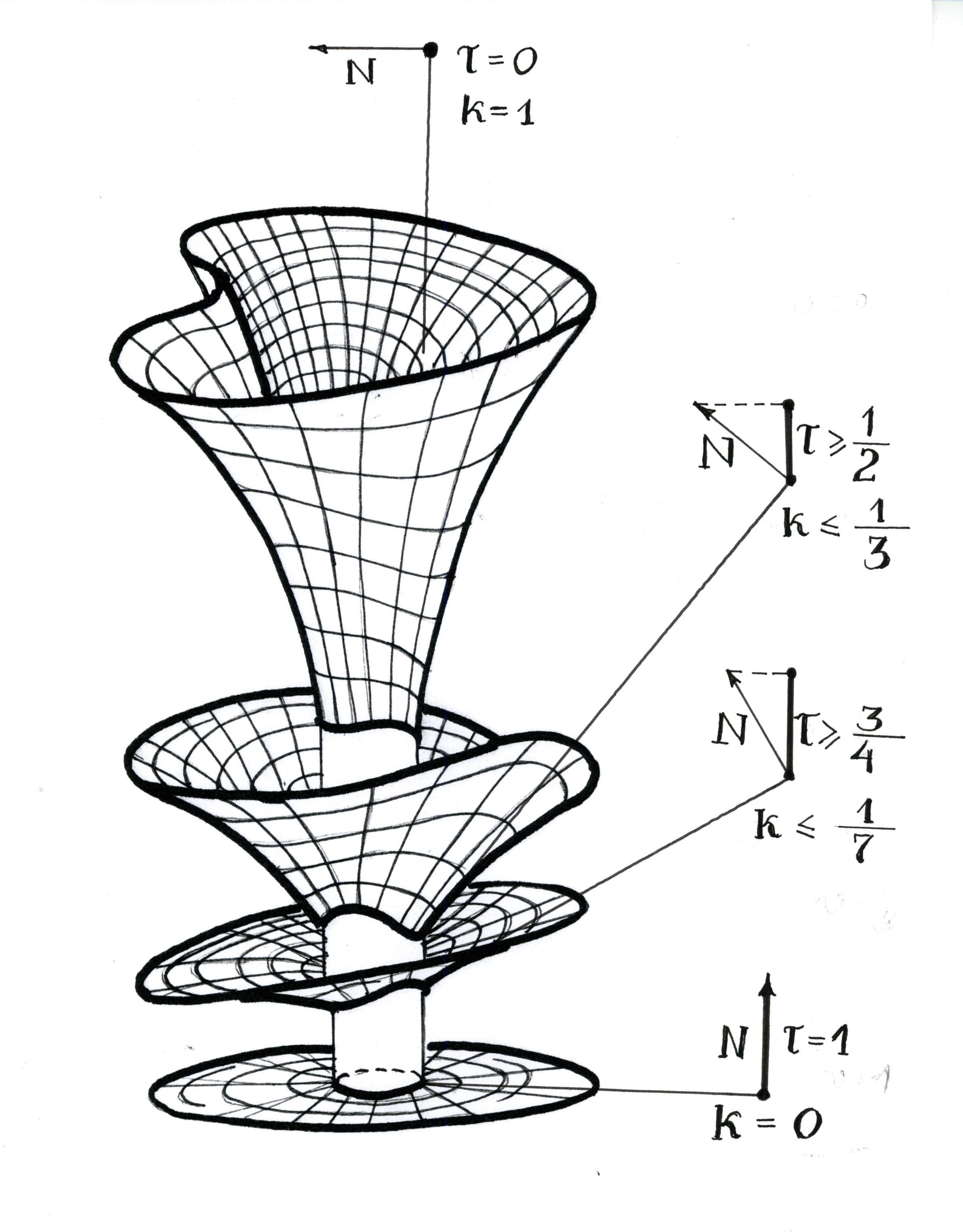}
\caption{Evolution of minimal surfaces with different initial slopes}
\end{figure}
\end{center}

One general method for obtaining Cauchy data is by cutting a hole in a given surface with the aid of a cylindrical chisel. This is the method we want to exemplify in detail, as it motivates geometrically our sharp estimates in Section~\ref{confmod}.

\subsection{Example: Enneper's surface evolves from hyperbolic paraboloid}

We start with the hyperbolic paraboloid
\[w=x^2-y^2=\re z^2\]
which is a negatively curved surface, but not minimal.
Take its intersection with $\abs{z}=1$ as initial data; that is,
\[F_\circ(e^{i\theta})=(e^{i\theta},\cos 2\theta);\quad N_\circ(e^{i\theta})=\frac{1}{\sqrt{5}}(-2e^{-i\theta},1)\]
Notice the constant slope of normal vector $N_\circ$; it corresponds to
\[k=\frac{\sqrt{5}-1}{\sqrt{5}+1} \,, \quad   K = \sqrt{5}\,,\quad \tau = \frac{1}{ \sqrt{5}}\]

To find the solution of the above Bj\"orling problem, we first find the antianalytic $\nu$-coefficient,
\[\nu(e^{i \theta})=-k\,e^{-2i\theta}, \quad \mbox{hence } \nu(z)= -k\bar z^2, \quad z \in \C \]
On the unit circle we have
\[
h_\rho+i h_\theta=\nu(\bar h_\rho+i\bar h_\theta)
\]
hence
\[
h_\rho= \,- i\, \frac{1+\abs{\nu}^2}{1-\abs{\nu}^2}\,h_\theta\;+\;\frac{2\nu\,i}{1-\abs{\nu}^2}\,\bar h_\theta
\]
Given our initial condition $h(e^{i\theta})=e^{i\theta}$,
\[
\begin{split}
h_\rho&=\frac{1+\abs{\nu}^2}{1-\abs{\nu}^2}e^{i\theta}+\frac{2\nu}{1-\abs{\nu}^2}e^{-i\theta} \\
&= \frac{3}{\sqrt{5}}e^{i\theta}-\frac{2}{\sqrt{5}}e^{-3i\theta}
\end{split}
\]
Knowing $h$ and $\,h_\rho\,$ on the unit circle determines uniquely its harmonic extension,
\[
h(\rho e^{i\theta})=\frac{\rho+\rho^{-1}}{2}e^{i\theta}+
\frac{3}{\sqrt{5}}\frac{\rho-\rho^{-1}}{2}e^{i\theta}-\frac{2}{\sqrt{5}}\frac{\rho^3-\rho^{-3}}{6}e^{-3i\theta}
\]
To find $w$, we turn to the first of the normality conditions in (\ref{ss}), namely:  $\re (\xi \bar h_\rho)+\tau w_\rho=0 $,
where
\[\xi=\frac{-2}{\sqrt{5}}e^{-i\theta},\quad \tau=\frac{1}{\sqrt{5}}\]
From this we find
\[w_\rho= \frac{2}{\sqrt{5}}\cos 2\theta\]
and therefore
\[
w(\rho e^{i\theta})=\frac{\rho^2+\rho^{-2}}{4}\cos 2\theta + \frac{\rho^2-\rho^{-2}}{4}\frac{2}{\sqrt{5}}\cos 2\theta
\]
The so obtained isothermal parameters $F=(h,w)$ represent familiar Enneper's surface that evolves from a
Jordan curve with the surface slope being constant along the curve.
In Figure~\ref{ennepprrr} we illustrate four stages of this evolution.
\begin{center}
\begin{figure}[h!]
\includegraphics[height=70mm]{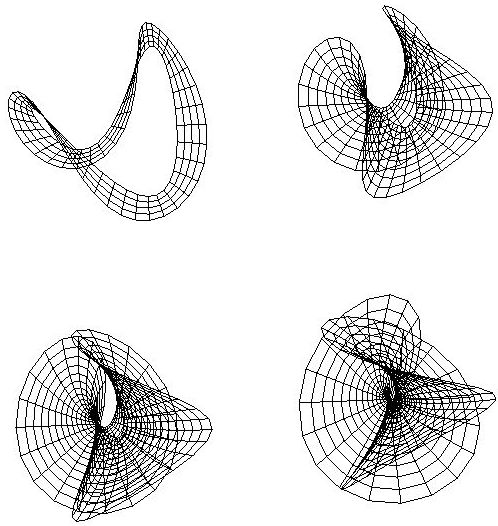}
\caption{Enneper's surface evolves from a closed curve}\label{ennepprrr}
\end{figure}
\end{center}

Perhaps the most natural way of imposing the Bj\"orling data is to borrow it from an existing doubly connected strip of a negatively curved surface,
as we did above with the hyperbolic paraboloid. This is reminiscent of a weak formulation of the Dirichlet problem in a domain $\Omega$ when
the boundary data is presented in the form of a function defined in $\Omega$.

\subsection{Conformal modulus of minimal surfaces}\label{confmod}
The following result is a reformulation of Theorem~\ref{th34} in terms of minimal graphs.
By a \emph{half-circular annulus} we mean a doubly connected domain $\mathcal A\subset \C$ whose inner boundary is the
 unit circle $\T= \left\{z \in \C \colon \abs{z}=1\right\}$.

\begin{theorem}\label{graph}
Let $\Sigma$ be a minimal graph represented by the function $w=f(u,v)$
that is $\mathscr C^1$-smooth in the closure of a half-circular annulus
$\mathcal{A}= \mathcal A(1, \cdot)\subset \{1<u^2+v^2 <\sigma^2\}$. Then
\begin{equation}\label{graph1}
\Mod \Sigma \le \log\frac{K\sigma +\sqrt{K^2\sigma^2-K^2+1}}{K+1},
\end{equation}
where $K\ge 1$ is defined by
\[K^2=1+\max_{u^2+v^2=1}\abs{\nabla f(u,v)}^2  \]
 \end{theorem}

In fact, Theorem~\ref{th34} yields the estimate~\eqref{graph1}
to minimal surfaces other than graphs.
Let us say that a surface $\widetilde{\Sigma}$ is an extension of $\Sigma$ if
$\widetilde{\Sigma}$ admits a parametrization that extends some parametrization of $\Sigma$.
Figure~\ref{ennepprrr} shows how a minimal graph $\Sigma$ (in the upper left corner) extends
to a minimal surface $\widetilde{\Sigma}$ that not only fails to be a graph, but also has self-intersections.

\begin{theorem}\label{nongraph}
Let $\Sigma$, $\sigma$ and $K$ be as in Theorem~\ref{graph}, and let $\widetilde{\Sigma}$ be a doubly connected
minimal surface that extends $\Sigma$. If the image of $\widetilde{\Sigma}$ is still  contained
in the cylinder
\[\{(u,v,w)\in\R^3\colon 1< u^2+v^2 < \sigma^2\}\]
then
\begin{equation}\label{nongraph1}
\Mod \widetilde{\Sigma} \le \log\frac{K\sigma +\sqrt{K^2\sigma^2-K^2+1}}{K+1},
\end{equation}
Equality is attained  if $\widetilde{\Sigma}$ is a catenoidal slab $F=(h,w)$ with
\[
h(z)=\frac{K+1}{2K}\,z+\frac{K-1}{2K}\frac{1}{\bar z}, \qquad
w(z)=\frac{\sqrt{K^2-1}}{K}\log\abs{z}
\]
\end{theorem}
\begin{proof}
Let $R= \Mod  \widetilde{\Sigma}$. Let $F \colon A[1,R] \to  \C \times \R$ be a isothermal parametrization of $ \widetilde{\Sigma}$ such that $F=(h,w)$ where $h$ maps $\T$ homeomorphically onto itself, preserving the orientation. The definition of $K$ implies~\eqref{311}. Inequality~\eqref{312} yields
\[\sigma \ge  \frac{K+1}{2K}\;R+ \frac{K-1}{2K}\;\frac{ 1}{R}\]
Solving for $R$ we arrive at~\eqref{nongraph1}.
\end{proof}

The reader will notice that the minimal surface that arises from the Bj\"orling problem in
Proposition~\ref{cauchyprop} satisfies the assumptions of Theorem~\ref{nongraph} provided that
$h_\circ$ is a sense-preserving self-homeomorphism of $\T$.

\section{Proof of Theorem~\ref{th34}}

Let us first dispose of the easy case $K=1$. Since $h$ is harmonic, its derivative $h_{\bar z}$ is an antianalytic function. The inequality~\eqref{311} implies that $h_{\bar z}$ vanishes on $\T$ and therefore it is identically zero. Thus this case of Theorem~\ref{th34} reduces to a version of Schottky's theorem, see Proposition 3.1 in~\cite{IKO2} and also~\cite{BPC}.

From now on $K>1$. Two integral inequalities for complex harmonic functions will come into play. The first of these inequalities applies for small values of $R$ and relates the integral means of $h$ and its derivatives
to the integral of a nonnegative function over the annulus $\A=A(1,R)$.

\begin{proposition}\label{proplesse}
Let $\lambda>-1$ and $1<R\le 1+\sqrt{3+3\lambda}$.
Suppose $h \colon A[1,R] \to \C$ is a $\CC^1$-smooth mapping that is harmonic in $A(1,R)$. Then
\begin{equation}\label{theidentity}
\begin{split}
 & \hskip-0.2cm \frac{2R^2}{R^2+\lambda} \dashint_{\T_R} \abs{h}^2  - 2\frac{\lambda R^2+1}{(1+\lambda)^2} \dashint_{\T} \abs{h}^2
-2\frac{R^2-1}{1+\lambda} \dashint_{\T} \abs{h} \abs{h}_\rho  \\
&-\frac{2(R-1)^2 (2R+3\lambda+1)}{3(1+\lambda)^2}\;  \dashint_{\T} \im [\bar h \left(h_\theta - i h\right)] \\
&   \ge  \frac{1}{\pi} \iint_{\A} \frac{ (R-\rho)^2 (2R\rho +\rho^2+3\lambda)}{3 \rho^2}          \cdot
\left|\frac{\rho h_\rho-ih_\theta}{\rho^2+\lambda} - \frac{2 \rho^2h}{(\rho^2+\lambda)^2}  \right|^2
\end{split}
\end{equation}
The righthand side vanishes if and only if $h$ is constant multiple of $h^\upsilon$ with
$\upsilon=\frac{1-\lambda}{1+\lambda}$; that is,
\begin{equation}\label{hequal}
h(z)=c\left(z+\frac{\lambda}{\bar z}\right),\quad c\in\C.
\end{equation}
\end{proposition}

Proposition~\ref{proplesse} will be used with $\lambda=\frac{K-1}{K+1}> 0$. Hence, it is applicable whenever $R \le 1 + \sqrt{3}$; the remaining values of $R$ are covered by the following result.

\begin{proposition}\label{bigrprop}
Let $0<\lambda \le 1$ and $R \ge \sqrt{7}$. Suppose $h \colon A[1,R] \to \C$ is a $\CC^1$-smooth mapping that is harmonic in $A(1,R)$.
 Denote by $f\colon \overline{\DD}\to\C$  the harmonic extension of $h$ to the closed unit disk. Then for all
$\sqrt{7}\le \rho\le R$ we have
\begin{equation}\label{trueforall}
\begin{split}
\dashint_{\T_\rho}\abs{h}^2 & -\left(\frac{\rho^2+\lambda}{(1+\lambda)\rho} \right)^2\dashint_{\T}\im (\bar h h_\theta)  - 2\dashint_{\T} \left\{\abs{h}\abs{h}_\rho -\frac{1-\lambda}{1+\lambda} \im (\bar h h_\theta)\right\}  \\ &  -\frac{\rho^2-(3+\lambda)-\lambda\rho^{-2}}{\lambda(1+\lambda)}\dashint_{\T} (\lambda^2\abs{h_z}^2-\abs{h_{\bar z}}^2)\\ &  -\frac{\rho^2-(3+\lambda)-\lambda\rho^{-2}}{(1+\lambda)} \left[\int_{\T}  J_f-  \iint_{\mathbb D} \abs{Df}^2\right]\ge 0
\end{split}
\end{equation}
The equality occurs if and only if $h$ is given by~\eqref{hequal}.
\end{proposition}

Lengthy computations for the  proofs of Propositions~\ref{proplesse} and~\ref{bigrprop} are given in  Section~\ref{propproofs}.

We now proceed to prove Theorem~\ref{th34}. We make frequent use of polar coordinates $z=\rho e^{i\theta}$.
Our first step is to prove  the inequality
\begin{equation}\label{texasrh}
\abs{h(z)}_\rho \ge  \frac{1}{K} \abs{h_\theta (z)}, \qquad z \in \T
\end{equation}
For this we compute the Jacobian $J_h(z)=\det Dh(z)$ for $z\in \T$ as follows.
\begin{equation*}
J_h=\im (\overline {h_\rho} h_\theta)
=\im \big[(\overline{ h_\rho} h) (\overline{h} h_\theta)\big]=\abs{h_\theta}\re (\overline{ h_\rho} f)=\abs{h_\theta}\abs{h}_{\rho}.
\end{equation*}
Combining this with~\eqref{311} we have
\[\abs{h_\theta}\abs{h}_{\rho} \ge \frac{1}{K} \norm{Dh}^2 \ge \frac{1}{K} \abs{h_\theta}^2  \]
which implies~\eqref{texasrh}.
Since the winding number $h$ on $\T$ is $1$, it follows
\begin{equation}\label{winding}
\dashint_{\T}\im (\bar h h_\theta) =1
\end{equation}

{\bf Case 1.} $R \le 1 + \sqrt{3}$. We set $\lambda=\frac{K-1}{K+1}$ in Proposition~\ref{proplesse}  and observe that
\begin{enumerate}[(i)]
\item $\ds \dashint_{\T} \abs{h}^2 =1$
\item $\ds \dashint_{\T} \abs{h} \abs{h}_\rho  \ge \frac{1}{K}=\frac{1-\lambda}{1+\lambda}$ by~\eqref{texasrh}
\item $\ds   \dashint_{\T} \im [\bar h \left(h_\theta - i h\right)] =0 $ by~\eqref{winding}
\end{enumerate}
Hence,~\eqref{theidentity} implies
\[ \frac{2R^2}{R^2+\lambda} \dashint_{\T_R} \abs{h}^2  \ge  2\frac{\lambda R^2+1}{(1+\lambda)^2} +2\frac{R^2-1}{1+\lambda} \frac{1-\lambda}{1+\lambda} = 2 \frac{R^2+\lambda}{(1+\lambda)^2} \]
from which~\eqref{312} follows.

Suppose the equality holds in~\eqref{312}. Then the righthand side of~\eqref{theidentity} must be also zero. Thus
$h$ is a constant multiple of $h^{\upsilon}$ with $\upsilon=1/K$. This finishes Case~$1$.

{\bf Case 2.} $R \ge \sqrt{7}$. Let $f$ be as in Proposition~\ref{bigrprop}.  We set $\lambda=\frac{K-1}{K+1}>0$ and observe that
\begin{enumerate}[(i)]
\item $\ds \dashint_{\T} \left\{\abs{h}\abs{h}_\rho -\frac{1-\lambda}{1+\lambda} \im (\bar h h_\theta)\right\} \ge 0$  by~\eqref{texasrh}
\item $\ds \dashint_{\T} (\lambda^2\abs{h_z}^2-\abs{h_{\bar z}}^2)$ by~\eqref{311}
\item $\ds  \int_{\T}J_f-\iint_{\mathbb D} \abs{Df}^2 \ge 0$
\end{enumerate}
The latter has been established in~\cite[Theorem 1.7]{IKO2} under the name Jacobian-Energy inequality, which we recall follows.
\begin{theorem}
Let a harmonic homeomorphism $f \colon \overline{\mathbb D}  \overset{\textnormal{\tiny{onto}}}{\longrightarrow} \overline{\mathbb D}$ be $\mathscr C^1$-smooth in the closed unit disk $\overline{\mathbb D} = \{z \in \C \colon \abs{z} \le 1\}$. Then
\begin{equation}\label{th3form}
\int_{\T} \abs{\det Df} \ge \iint_{\mathbb D} \abs{Df}^2
\end{equation}
The  inequality is strict unless  $f$ is an isometry.
\end{theorem}
Now,~\eqref{trueforall} yields
\[\dashint_{\T_R}\abs{h}^2 \ge \left(\frac{R^2+\lambda}{(1+\lambda)R} \right)^2  \]
which is~\eqref{312}.

If equality holds in~\eqref{312}, then it also holds in~\eqref{trueforall}. Thus
$h$ is a constant multiple of $h^{\upsilon}$ with $\upsilon=1/K$. This completes the proof of Theorem~\ref{th34}
modulo the two propositions.
\qed

\section{Proof of Propositions~\ref{proplesse} and~\ref{bigrprop}}\label{propproofs}

\begin{proof}[Proof of Proposition~\ref{proplesse}]
Inequality~\eqref{theidentity} is best interpreted in terms of integral means
\[U(\rho) = \dashint_{\T_\rho} \abs{h}^2\]
and the derivative of $U$, denoted $\dot{U}$. Indeed, the first three terms in~\eqref{theidentity} are nothing but
\begin{equation}\label{first3}
\frac{2R^2}{R^2+\lambda} U(R)- 2\frac{\lambda R^2+1}{(1+\lambda)^2}U(1)-\frac{R^2-1}{1+\lambda}\dot{U}(1)
\end{equation}
There is a way to express~\eqref{first3} as a double integral over $A(1,R)$, see~(8.5) and~(8.14) in~\cite{IKO}.
This integral takes a simpler form in terms of the function~$g=h/h^\upsilon$.
\begin{equation}\label{equ4}
\begin{split}
\frac{2R^2}{R^2+\lambda} U(R)&- 2\frac{\lambda R^2+1}{(1+\lambda)^2}U(1)-\frac{R^2-1}{1+\lambda}\dot{U}(1)\\
&= \frac{4}{(1+\lambda)^2}\int_1^R \frac{(R^2-\rho^2)(\rho^2+\lambda)}{\rho} \dashint_{\T_\rho} \left(  \abs{g_z}^2 + \abs{g_{\bar z}}^2 \right)\\ &\; \; +  \frac{4}{(1+\lambda)^2}\int_1^R \frac{(R^2-\rho^2)(\rho^2+\lambda)}{\rho^3} \dashint_{\T_\rho}\im \left(  \bar g \, g_\theta \right)
\end{split}
\end{equation}
The next step in~\cite{IKO} was to perform integration by parts in the second term, but here we treat it differently, by splitting it as
\[
\frac{(R^2-\rho^2)(\rho^2+\lambda)}{\rho^3}=\alpha- \frac{\dtext \beta}{\dtext \rho}
\]
where
\[\begin{split}
\alpha&= \frac{(R-\rho)(2\rho^2+2R\rho+3\lambda-R^2)}{3\rho^2} \\
\beta&=  \frac{(R-\rho)^2 (2R\rho +\rho^2+3\lambda)}{6\rho^2}
\end{split} \]
The term with $\beta$ is integrated by parts,
\[
\begin{split}
\int_1^R - \frac{\dtext \beta}{\dtext \rho} \dashint_{\T_\rho}\im \left(  \bar g \, g_\theta \right) & = \int_1^R \beta  \frac{\dtext}{\dtext \rho} \dashint_{\T_\rho}\im \left(  \bar g \, g_\theta \right) \\ &+ \frac{(R-1)^2(2R+1+3\lambda)}{6} \dashint_{\T} \im \left( \bar g \, g_\theta \right)
\end{split}
\]
Another integration by parts, this time along $\T_\rho$, gives
\[
\begin{split}
\frac{\dtext}{\dtext \rho} \dashint_{\T_\rho} \im (\bar g \, g_\theta) &= \im \dashint_{\T_\rho} \left( \bar g_\rho g_\theta + \bar g g_{\rho \theta} \right) \\
&= \im \dashint_{\T_\rho} \left( \bar g_\rho g_\theta - \bar g_\theta  g_{\rho} \right) = 2 \rho \dashint_{\T_\rho} \left(  \abs{g_z}^2 - \abs{g_{\bar z}}^2\right)
\end{split}
\]
The condition $R \le 1+ \sqrt{3+3\lambda}$ ensures that $\alpha \ge 0$. The integral with $\alpha$ can be estimated
using the Cauchy-Schwarz and Wirtinger inequalities,
\begin{equation}\label{alpha}
\begin{split}
\bigg|\dashint_{\T_\rho}\im ( \bar g \, g_\theta) \bigg|
&\le \bigg(\dashint_{\T_\rho} \bigg|g-\dashint_{\T_\rho} g\bigg|^2\bigg)^{1/2} \bigg(\dashint_{\T_\rho}\abs{g_\theta}^2\bigg)^{1/2}\\
&\le \dashint_{\T_\rho}\abs{g_\theta}^2
 \le 2\rho^2 \dashint_{\T_\rho}\left(  \abs{g_z}^2 + \abs{g_{\bar z}}^2\right)
\end{split}
\end{equation}
Therefore, the righthand of~\eqref{equ4} is estimated from below as
\[
\begin{split}
&\frac{4}{(1+\lambda)^2}\int_1^R  \left[\frac{(R^2-\rho^2)(\rho^2+\lambda)}{\rho} +  2\rho  \beta
- 2 \rho^2 \alpha \right] \dashint_{\T_\rho} \abs{g_z}^2\\
&  + \frac{4}{(1+\lambda)^2}\int_1^R  \left[\frac{(R^2-\rho^2)(\rho^2+\lambda)}{\rho} -    2\rho  \beta - 2 \rho^2 \alpha \right] \dashint_{\T_\rho} \abs{g_{\bar z}}^2\\
&\hskip2cm  +\frac{4}{(1+\lambda)^2}\frac{(R-1)^2(2R+1+3\lambda)}{6} \dashint_{\T} \im \left( \bar g \, g_\theta \right)
\end{split}
\]
which simplifies as
\[
\begin{split}
\frac{4}{(1+\lambda)^2}&\int_1^R  \frac{2 (R-\rho)^2 (2R\rho +\rho^2+3\lambda)}{3 \rho} \dashint_{\T_\rho} \abs{g_z}^2 \\
&+\frac{2(R-1)^2 (2R+3\lambda+1)}{3(1+\lambda)^2} \dashint_{\T} \im \left( \bar g \, g_\theta \right)
\end{split}
\]
To conclude with the equality~\eqref{theidentity} it only remains to observe that
\[
\im \left( \bar g \, g_\theta \right) = \frac{(1+\lambda)^2\rho^2}{(\rho^2+\lambda)^2} \im \big(\bar h h_\theta-i\bar hh\big)
\]
\begin{equation}\label{gzgz}
\abs{g_z} =  \frac{1+\lambda}{2}\left|\frac{\rho h_\rho-ih_\theta}{\rho^2+\lambda} - \frac{2 \rho^2h}{(\rho^2+\lambda)^2}  \right|
\end{equation}
This finishes the proof of Proposition~\ref{proplesse}, except for the equality statement.
The righthand side of~\eqref{theidentity} vanishes if and only if $g_z\equiv 0$, due to~\eqref{gzgz}.
Recall that $h=h^\upsilon g$ where both $h$ and $h^\upsilon$ are harmonic. Hence
\[0\equiv h_{z\bar z}=(h^\upsilon_z g)_{\bar z}=h^\upsilon_z g_{\bar z}.\]
Since $h^\upsilon_z\equiv (1+\lambda)^{-1}\ne 0$, we have $g_{\bar z}\equiv 0$.
Thus $g$ is a constant function, and the proof is complete.
\end{proof}

Finally we turn to Proposition~\ref{bigrprop}, which will make the preceding proof appear short and elegant.

\begin{proof}[Proof of Proposition~\ref{bigrprop}]
Since $h$ is harmonic in a circular annulus, it admits an expansion
\begin{equation}\label{hseries}
h(z)=a_0\log\abs{z}+b_0+\sum_{n\ne 0}(a_n z^n + b_n \bar z^{-n})
\end{equation}
for some $a_n, b_n\in\C$. Thanks to Parseval's identity the proof reduces to elementary manipulations with Fourier coefficients $a_n, b_n$. Indeed, upon substituting formulas
\begin{equation}\label{mess}
\begin{split}
\dashint_{\T_\rho} \abs{h}^2 &= \abs{a_0 \log \rho +b_0}^2 + \sum_{n \ne 0} \abs{a_n\rho^n + b_n \rho^{-n}}^2;\\
\dashint_{\T} \abs{h}\abs{h}_\rho &= \frac{1}{2}\, \dashint_{\T} \abs{h^2}_\rho =  \re( a_0\bar b_0)+ \sum_{n \ne 0} n(\abs{a_n}^2-\abs{b_n}^2);
\\
\dashint_{\T} \im \left( \bar h \, h_\theta \right) &= \sum_{n \ne 0} n \abs{a_n+b_n}^2; \hskip0.5cm     \dashint_{\T} J_h = \sum_{n \ne 0} n^2(\abs{a_n}^2-\abs{b_n}^2) ;\\
\dashint_{\T} J_f &= \sum_{n \ne 0} n\abs{n} \abs{a_n+b_n}^2; \hskip0.5cm
\iint_{\mathbb D} \abs{Df}^2 = 2\pi \sum_{n \ne 0} \abs{n} \abs{a_n+b_n}^2
\end{split}
\end{equation}
into~\eqref{trueforall} the lefthand side is represented by the quadratic form
\[Q= \sum_{n \in \Z} Q_n (a_n, b_n)\, , \quad Q_n(\xi, \zeta)= A_n \abs{\xi}^2 + B_n \abs{\zeta}^2 + 2 \, C_n \re (\xi \, \bar \zeta)\]
Precisely, we have for $n \ne 0$
\[
\begin{split}
Q_n (\xi, \zeta)&= \abs{\rho^n \xi + \rho^{-n} \zeta}^2- \left[\frac{(\rho+\lambda/\rho)^2}{(1+\lambda)^2 } - 2 \frac{1-\lambda}{1+\lambda}\right] n  \abs{\xi+\zeta}^2 -2n \left(\abs{\xi}^2 - \abs{\zeta}^2\right)\\ &- \frac{\rho^2-(3+\lambda)-\lambda\rho^{-2}}{\lambda(1+\lambda) } \left[n^2 \left(\lambda^2\abs{\xi}^2 - \abs{\zeta}^2\right)+ \lambda \abs{n}(n-1) \abs{\xi+\zeta}^2  \right]
\end{split}
\]
and
\[Q_0(\xi, \zeta)= \abs{\xi \log \rho + \zeta}^2 -2\re (\xi \, \bar \zeta)\]
The form $Q_0$ is positive definite provided $\rho > \sqrt{e}$. The identity
\[Q_1(\xi, \zeta)= \frac{\rho^2-(3-\lambda^2)  +\lambda^2\rho^{-2}  }{\lambda (1+\lambda)^2} \abs{\lambda \xi - \zeta}^2 \]
shows that $Q_1$ is positive semidefinite. Proposition~\ref{bigrprop}, together with the equality
statement, will follow once we show that the forms $Q_n(\xi, \zeta)$ are positive definite for
$n \ge 2$ and for $n \le -1$. For this we must prove
\begin{equation}\label{goal101}
A_nB_n>C_n^2, \qquad \text{and }\ A_n,B_n>0
\end{equation}
Explicitly,
\begin{equation*}\label{coef1}
\begin{split}
A_n&=  \rho^{2n} - \left[\frac{(\rho+\lambda/\rho)^2}{(1+\lambda)^2 } - 2 \frac{1-\lambda}{1+\lambda}\right] n  -2n \\ & \phantom{= \rho^{2n}\,\,} - \frac{\rho^2-(3+\lambda)-\lambda\rho^{-2}}{\lambda(1+\lambda) } \left[n^2 \lambda^2 + \lambda \abs{n}(n-1)  \right]; \\
B_n&= \rho^{-2n} - \left[\frac{(\rho+\lambda/\rho)^2}{(1+\lambda)^2 } - 2 \frac{1-\lambda}{1+\lambda}\right] n  +2n\\ & \phantom{= \rho^{-2n}\,\, } - \frac{\rho^2-(3+\lambda)-\lambda\rho^{-2}}{\lambda(1+\lambda) } \left[-n^2 + \lambda \abs{n}(n-1)  \right]; \\
C_n&=1- \left[\frac{(\rho+\lambda/\rho)^2}{(1+\lambda)^2 } - 2 \frac{1-\lambda}{1+\lambda}\right] n - \frac{\rho^2-(3+\lambda)-\lambda\rho^{-2}}{ 1+\lambda }  \abs{n}(n-1)  .
\end{split}
\end{equation*}
Our proof of~\eqref{goal101}, while elementary, is somewhat lengthy.
We split it into three cases: $n\ge 3$, $n=2$, and $n\le -1$.
Inequalities $\rho^2\ge 7$ and $0<\lambda\le 1$ will be used repeatedly without mention.

\begin{center}
\textsc{Case  $n \ge 3$}
\end{center}

Inequality~\eqref{goal101} for $n\ge 3$ is an immediate consequence of the following estimates:
\begin{align}
A_n & \ge (7^{n-1}-n^2)\rho^2 \label{Aineq1} \\
B_n &\ge \frac{31}{28}\frac{n\rho^{2}}{7\lambda} \label{Bineq1} \\
0\le -C_n &\le n^2\rho^2 \left(\frac{1}{1+\lambda}-\frac{\lambda}{n(1+\lambda)^2}\right)  \label{Cineq1} \\
\frac{31}{28} \left(\frac{7^{n-1}}{n^3} - \frac{1}{n}\right) &> 7 \lambda \left( \frac{1}{1+\lambda} - \frac{\lambda}{n(1+\lambda)^2}  \right)^2 \label{Dineq1}
\end{align}

\textsc{Proof of~\eqref{Aineq1}.} We write $A_n= \rho^{2n}+ \alpha_2 \rho^2 + \alpha_0 + \alpha_{-2} \rho^{-2}$, where the coefficients  depend on $n$ as well as on $\rho$ and $\lambda$. Specifically,
\begin{equation}\label{alphas}
\begin{split}
\alpha_2&=\frac{\lambda n}{(1+\lambda)^2}-n^2 \\
\alpha_{-2}&=\frac{1}{1+\lambda}\left[-\frac{\lambda^2n}{1+\lambda}+n^2\lambda+\lambda n(n-1)\right]\ge 0 \\
\alpha_0&=-\frac{2\lambda n}{(1+\lambda)^2} +2n\frac{1-\lambda}{1+\lambda}-2n +\frac{3+\lambda}{1+\lambda}(n^2\lambda+n(n-1))\ge 0
\end{split}
\end{equation}
where the positivity of $\alpha_0$ can be observed by rearranging it as
\[
\begin{split}
\alpha_0&=2n\frac{1-\lambda}{1+\lambda} + \frac{\lambda n}{1+\lambda}\left\{(3+\lambda)n-\frac{2}{1+\lambda}\right\}\\
&+\frac{n}{1+\lambda}\left\{(3+\lambda)(n-1)-2(1+\lambda)\right\}
\end{split}
\]
Thus
\begin{equation}\label{An1}
A_n\ge \rho^{2n} + \alpha_2 \rho^2 = (\rho^{2n-2}-n^2)\rho^2\ge (7^{n-1}-n^2)\rho^2.
\end{equation}

\textsc{Proof of~\eqref{Bineq1}.} Ignoring the positive term $\rho^{-2n}$ in $B_n$ and using the inequality
$\lambda n(n-1)\le n(n-1)$, we obtain the desired estimate
\[
\begin{split}
B_n&\ge
 - \left[\frac{(\rho+\lambda/\rho)^2}{(1+\lambda)^2 } - 2 \frac{1-\lambda}{1+\lambda}\right] n  +2n + n\frac{\rho^2-(3+\lambda)-\lambda\rho^{-2}}{\lambda(1+\lambda) } \\
 &=\frac{n\rho^{2}}{\lambda} \left(\frac{1 -(3-\lambda^2) \rho^{-2}}{(1+\lambda)^2}- \frac{\lambda (1+\lambda+\lambda^2)}{\rho^{4}(1+\lambda)^2}\right) \ge \frac{n\rho^{2}}{\lambda } \left(\frac{4+\lambda^2}{7(1+\lambda)^2}- \frac{1}{49}\right)\\ &=\frac{n\rho^{2}}{7\lambda } \left(\frac{4+\lambda^2}{(1+\lambda)^2}- \frac{1}{7}\right)
= \frac{n\rho^{2}}{7\lambda } \left(1+ \frac{3-2\lambda}{(1+\lambda)^2}- \frac{1}{7}\right)
\ge \frac{31}{28}\frac{n\rho^{2}}{7\lambda }
\end{split}
\]

\textsc{Proof of~\eqref{Cineq1}.} The first inequality in~\eqref{Cineq1} follows from
\[
C_n\le 1- \frac{n}{1+\lambda}\left[\frac{(\rho+\lambda/\rho)^2}{1+\lambda} - 2(1-\lambda)\right]
\le 1-\left[\frac{7}{2} - 2\right]<0.
\]
Writing $C_n= \gamma_2\rho^2 + \gamma_0+\gamma_{-2}\rho^{-2}$, we find
\begin{equation}\label{gammas}
\begin{split}
\gamma_2&=  n^2 \left\{ \frac{\lambda}{n(1+\lambda)^2}- \frac{1}{1+\lambda}\right\} \le 0 \\
\gamma_0&= 1+\frac{n}{1+\lambda}\left\{\frac{-2\lambda}{1+\lambda}+2(1-\lambda)+(3+\lambda)(n-1)\right\}\ge 0 \\
\gamma_{2}& = \frac{n}{1+\lambda}\left\{\lambda(n-1)-\frac{\lambda^2}{1+\lambda}\right\}\ge 0
\end{split}
\end{equation}
Hence $-C_n \le - \gamma_2(n)\rho^2$ as desired.

\textsc{Proof of~\eqref{Dineq1}.} When $n \ge 4$, we use the monotonicity of the lefthand side of~\eqref{Dineq1} to obtain
\[\left(\frac{7^{n-1}}{n^3} - \frac{1}{n}\right)> \frac{7^3}{4^3}-\frac{1}{4}> \frac{7}{4} \ge \frac{7\lambda}{(1+\lambda)^2},
\]
which implies~\eqref{Dineq1}. When $n=3$, the inequality~\eqref{Dineq1} takes form
\[ \frac{ \lambda (3+2 \lambda)^2}{(1+\lambda)^4}< \frac{310}{147}\]
This holds even if the fraction on the right is replaced with $2$, because
\[
\lambda (3+2 \lambda)^2 - 2(1+\lambda)^4 = -2 + \lambda -4\lambda^3-\lambda^4<0
\]

\begin{center}
\textsc{Case $n=2$}
\end{center}

Compared to the case $n\ge 3$, we will have to keep more terms in the estimates for $A_n$ and $C_n$. It follows from~\eqref{alphas} and~\eqref{gammas} that
\[A_2\ge \rho^4+\alpha_2\rho^2 \quad \text{ and } \quad \abs{C_2}\le -\gamma_2\rho^2-\gamma_0\]
By~\eqref{Bineq1}, $B_2\ge \beta \rho^2$ with $\beta =\frac{31}{98\lambda}$.
Together, these estimates yield
\[
\begin{split}
A_2B_2-C_2^2 &\ge \beta \rho^2 (\rho^4+\alpha_2\rho^2)- (\gamma_2\rho^2+\gamma_0)^2\\ &
= \beta \rho^6+ (\beta \alpha_2 - \gamma_2^2) \rho^4 - 2 \gamma_0\gamma_2 \rho^2 - \gamma_0^2
\end{split}
\]
Dividing by $\rho^2$ and noticing that $- \gamma_0^2/\rho^{2} \ge - \gamma_0^2/7$, we reduce the task to showing
that
\begin{equation}\label{goal100}
\beta \rho^4 + (\beta \alpha_2 - \gamma_2^2) \rho^2 - 2 \gamma_0\gamma_2  - \frac{\gamma_0^2}{7} > 0
\end{equation}
Substitution $\rho^2=t+7$ turns the lefthand side of~\eqref{goal100} into a quadratic polynomial in $t$ with
leading coefficient $\beta>0$. It remains to show that the coefficients of $t^1$ and $t^0$ are also positive, i.e.,
\begin{align}
&14 \beta + \beta \alpha_2 - \gamma_2^2 >0 \label{goal99}\\
&49 \beta + 7 (\beta \alpha_2 - \gamma_2^2)  - 2 \gamma_0\gamma_2  - \frac{\gamma_0^2}{7}  >0 \label{goal98}
\end{align}

\textsc{Proof of~\eqref{goal99}.} Recall that
\begin{align*}
\alpha_2&= \frac{2 \lambda}{(1+\lambda)^2}-4 \\
\gamma_2&= \frac{2\lambda}{(1+\lambda)^2} - \frac{4}{1+\lambda} \\
\gamma_0&= \frac{1}{1+\lambda} \left( 11-\lambda - \frac{4\lambda}{1+\lambda}  \right)
\end{align*}
Since $\beta>\frac{2}{7 \lambda}$, it follows that
\[
\begin{split}
 \beta(14 + \alpha_2) - \gamma_2^2& \ge  \frac{2}{7 \lambda} \left( 10 + \frac{2\lambda}{(1+\lambda)^2}\right) - \left(  \frac{4}{1+\lambda} - \frac{2\lambda}{(1+\lambda)^2}   \right)^2\\
& = \frac{1}{7(1+\lambda)^2} \left[  \frac{20}{\lambda} + 44 + 20 \lambda - 7 \left( 4- \frac{2\lambda}{1+\lambda}  \right)^2     \right]
\end{split}
\]
Next, we estimate $4- \frac{2\lambda}{1+\lambda} $ from above by $4- \lambda$ and then use the arithmetic-geometric mean inequality $x+y\ge \sqrt{4xy}$.
\[
\begin{split}
\beta(14+ \alpha_2) - \gamma_2^2& \ge  \frac{1}{7(1+\lambda)^2} \left[  \frac{20}{\lambda} + 76 \lambda - 68- 7 \lambda^2   \right] \\
& \ge \frac{1}{7(1+\lambda)^2} \left( \sqrt{80\cdot 76} - 75 \right) >0
\end{split}
\]

\textsc{Proof of ~\eqref{goal98}.} First complete the square and rearrange the terms as follows.
\[
\begin{split}
49 \beta &+ 7 (\beta \alpha_2 - \gamma_2^2)  - 2 \gamma_0\gamma_2  - \frac{\gamma_0^2}{7} = 7 \beta (7 + \alpha_2) -\frac{1}{7} \left(7\gamma_2+ \gamma_0 \right)^2 \\
&= \frac{1}{14 (1+\lambda)^2} \left[ \frac{93}{\lambda} + 248 + 93 \lambda - 2 \left( 17 + \lambda - \frac{10 \lambda}{1+\lambda}  \right)^2    \right]
\end{split}
\]
Similar to the proof of~\eqref{goal99}, we estimate $17 + \lambda - \frac{10 \lambda}{1+\lambda} $ from above by
$17-4 \lambda$ and finally use the arithmetic-geometric mean inequality.
\[
\begin{split}
49 \beta + 7 (\beta \alpha_2 - \gamma_2^2)  - 2 \gamma_0\gamma_2  - \frac{\gamma_0^2}{7} & \ge
\frac{1}{14 (1+\lambda)^2} \left[  \frac{93}{\lambda} +365 \lambda -330  - 32 \lambda^2   \right] \\
& \ge \frac{1}{14 (1+\lambda)^2} \left[   \sqrt{372 \cdot 365} - 362 \right]>0
\end{split}
\]

\begin{center}
\textsc{Case  $n \le -1$}
\end{center}

For convenience we set  $n=-m$ where $m$ is a positive integer. Inequality~\eqref{goal101} for $n\le -1$ is a direct consequence of the estimates
\begin{align}
A_{-m}&\ge \frac{3m\rho^2}{2(1+\lambda)} \label{Aineq2} \\
B_{-m}&\ge 7^{m-1}\rho^2+ \frac{m\rho^2}{1+\lambda}\left[ \frac{1}{1+\lambda} + \frac{11}{49}m + \frac{4m}{7 \lambda} \right]
\label{Bineq2} \\
0\le C_{-m} &\le \frac{m\rho^2}{(1+\lambda) } \left( m+\frac{3}{7}+ \frac{1}{1+\lambda}\right) \label{Cineq2} \\
\frac{2}{3} \left(   1+ \frac{1}{m} \left[ \frac{3}{7}+ \frac{1}{1+\lambda} \right]   \right)^2
&< \frac{7^{m-1}}{m^3}(1+\lambda)+ \frac{1}{m^2} \left[ \frac{1}{1+\lambda} + \frac{11}{49}m + \frac{4m}{7 \lambda} \right] \label{Dineq2}
\end{align}

Before proving~\eqref{Aineq2}--\eqref{Dineq2} we write down the coefficients of $Q_{n}$ in terms of $m$.
\begin{align*}
A_{-m}&=  \rho^{-2m} + m\left[\frac{(\rho+\lambda/\rho)^2}{(1+\lambda)^2 } - 2 \frac{1-\lambda}{1+\lambda}\right]   +2m \\
&\phantom{=  \rho^{-2m}\,\,} - \frac{\rho^2-(3+\lambda)-\lambda\rho^{-2}}{\lambda(1+\lambda) } \left[m^2 \lambda^2 - \lambda m(m+1)  \right]; \\
B_{-m}&= \rho^{2m} +m \left[\frac{(\rho+\lambda/\rho)^2}{(1+\lambda)^2 } - 2 \frac{1-\lambda}{1+\lambda}\right]   -2m\\
&\phantom{=  \rho^{2m}\,\,}  + \frac{\rho^2-(3+\lambda)-\lambda\rho^{-2}}{\lambda(1+\lambda) } \left[m^2 + \lambda m (m+1)  \right]; \\
C_{-m}&=1 + m\left[\frac{(\rho+\lambda/\rho)^2}{(1+\lambda)^2 } - 2 \frac{1-\lambda}{1+\lambda}\right]  + \frac{\rho^2-(3+\lambda)-\lambda\rho^{-2}}{ 1+\lambda }  m(m+1).
\end{align*}

\textsc{Proof of~\eqref{Aineq2}.}
Ignoring the term $\rho^{-2m}$ in $A_{-m}$ and using
\[\left[m^2 \lambda^2 - \lambda m(m+1)  \right]\le -\lambda m,\] we obtain the estimate
\[
\begin{split}
A_{-m}&\ge \frac{m}{1+\lambda}\left[\frac{(\rho+\lambda/\rho)^2}{1+\lambda}- 2(1-\lambda) +2 (1+\lambda)
+\rho^2-(3+\lambda)-\lambda\rho^{-2}\right] \\
&= \frac{3m\rho^2}{2(1+\lambda)}+ \frac{m}{(1+\lambda)^2}
\left[\frac{1-\lambda}{2}\rho^2+2\lambda-3+3\lambda^2-\lambda^2\rho^{-2}\right]
\end{split}
\]
The term in brackets is positive since
\[
\frac{7-7\lambda}{2}+2\lambda-3+3\lambda^2-\frac{\lambda^2}{7}\ge \frac{1-3\lambda+5\lambda^2}{2}>0
\]
This yields~\eqref{Aineq2}.

\textsc{Proof of~\eqref{Bineq2}.}
Let us write $B_{-m}= \rho^{2m}+ \beta_2 \rho^2 + \beta_0 + \beta_{-2} \rho^{-2}$.
We will show that $\beta_0$ and $\beta_{-2}$ are negative, and estimate $B_{-m}$ from below by
\begin{equation}\label{est7}
B_{-m} \ge \rho^2\bigg(7^{m-1}+ \beta_2+\frac{\beta_0}{7}+\frac{\beta_{-2}}{49}\bigg)
\end{equation}
Indeed,
\[\beta_{-2}=    \frac{m}{1+\lambda}  \left[   \frac{\lambda^2}{1+\lambda }   - \left(m + \lambda (m+1) \right) \right]\]
is clearly negative and can be estimated as
\begin{equation}\label{beta3}
\beta_{-2} \ge - \frac{3m^3}{1+\lambda}
\end{equation}
Also,
\[\beta_0= \frac{m}{1+\lambda} \left[  \frac{2 \lambda}{1+\lambda} -4- \frac{3+\lambda}{\lambda}m-( 3+\lambda) (m+1) \right]\]
is negative and  can be estimated as
\begin{equation}\label{beta2}
\beta_0 \ge  \frac{m}{1+\lambda} \left[  -7 -5m - \frac{3}{\lambda }m  \right]
\end{equation}
We also have
\begin{equation}\label{beta1}
\beta_2= \frac{m}{1+\lambda} \left[ \frac{1}{1+\lambda} + m+1 + \frac{m}{\lambda} \right]
\end{equation}
Formulas~\eqref{est7}--\eqref{beta1} yield~\eqref{Bineq2}.

\textsc{Proof of~\eqref{Cineq2}.} The first inequality in~\eqref{Cineq2} follows from
\[
C_{-m}\ge \frac{m}{1+\lambda}\left[\frac{(\rho+\lambda/\rho)^2}{1+\lambda} - 2(1-\lambda)\right]
\ge m\left[\frac{7}{2} - 2\right]>0
\]
Writing $C_{-m}= \gamma_2\rho^2 + \gamma_0+\gamma_{-2}\rho^{-2}$, we observe that
\[
\begin{split}
\gamma_{-2}&=\frac{m}{1+\lambda}\left[\frac{\lambda^2}{1+\lambda}  - \lambda (m+1)\right]\le 0\\
\gamma_0&=1 +
\frac{m}{1+\lambda}\left[\frac{2\lambda}{1+\lambda} - 2(1-\lambda) -(3+\lambda)(m+1)  \right]
\\& \le  1 +
\frac{m}{1+\lambda}\left[4\lambda-2 -2(3+\lambda)\right]
\le  m- \frac{6m}{1+\lambda} \le - \frac{4m}{1+\lambda}\\
\gamma_2&= \frac{m}{(1+\lambda) } \left( m+1+ \frac{1}{1+\lambda}   \right)
\end{split}\]
Hence
\[
C_{-m} \le \gamma_2\rho^2+\gamma_0 \le \rho^2(\gamma_2+\gamma_0/7)
\le \frac{m\rho^2}{(1+\lambda)} \bigg( m+\frac{3}{7}+ \frac{1}{1+\lambda} \bigg)
\]

\textsc{Proof of~\eqref{Dineq2}.}
When $m \ge 3$, the righthand side of~\eqref{Dineq2} is greater than $7^2/3^3=49/27$ while the lefthand side is at most
\[   \frac{2}{3} \left(   1+ \frac{1}{3} \left[ \frac{3}{7}+ 1 \right]   \right)^2 =
\frac{2}{3} \left(   1+ \frac{10}{21}   \right)^2<\frac{2}{3} \left(\frac{3}{2}\right)^2=\frac{3}{2}\]

When $m=2$, the lefthand side of~\eqref{Dineq2} is at most $96/49$ which is its value at $\lambda=0$. On the righthand side we have
\[ \begin{split} \frac{7}{8} (1+\lambda) &+    \frac{1}{4}   \left[ \frac{1}{1+\lambda} + \frac{22}{49} + \frac{8}{7 \lambda} \right] \ge \frac{7}{8} + \frac{7}{8} \lambda + \frac{1}{4}   \left[ \frac{1}{2} +0+  \frac{8}{7 \lambda} \right] \\
&= 1+ \frac{7}{8}\lambda + \frac{2}{7 \lambda} \ge 1+1=2
\end{split}
\]

When $m=1$, we rearrange the terms of~\eqref{Dineq2} so that it reads as
\begin{equation}\label{goal7}
\frac{2}{3} \left( \frac{10}{7}+ \frac{1}{1+\lambda}    \right)^2  -\frac{1}{1+\lambda} - \frac{60}{49}
< \lambda+        \frac{4}{7 \lambda}
\end{equation}
If $\lambda \le 1/4$, then
\[
\frac{2}{3} \left( \frac{10}{7}+ \frac{1}{1+\lambda}    \right)^2  -\frac{1}{1+\lambda} - \frac{60}{49}
\le  \frac{2}{3} \left( \frac{10}{7}+ 1 \right)^2  - \frac{4}{5}  - \frac{60}{49}<2 <\frac{4}{7\lambda}
\]
thus~\eqref{goal7} holds. For $\lambda > 1/4$ we obtain~\eqref{goal7} as follows,
\[
\begin{split}
\frac{2}{3} & \left\{ \left( \frac{10}{7}+ \frac{1}{1+\lambda}    \right)^2  -\frac{3}{2(1+\lambda)} - \frac{90}{49} \right\}\le  \frac{2}{3}  \left\{ \frac{10}{49}+ \frac{33}{14(1+\lambda)} \right\}  \\
& \le
\frac{2}{3}  \left\{ \frac{1}{4}+ 2 \right \} = \frac{3}{2} < \frac{4}{\sqrt{7}} \le \lambda+\frac{4}{7 \lambda}
\end{split} \]
where the last step is the arithmetic-geometric mean inequality.
\end{proof}

\section*{Acknowledgments}

Part of this research was carried out while two of the authors were visiting University of Helsinki.
They thank the university for its hospitality.
Figure~\ref{ennepprrr} was made using Maple~12.

\bibliographystyle{amsplain}

\end{document}